\documentclass{article}
\usepackage{amsmath, amsthm, amssymb, amsfonts}
\usepackage[dvipsnames]{xcolor}
\usepackage{tikz,tikz-network,tikz-cd}
\usepackage{tcolorbox}
\usepackage{graphicx}
\usepackage{physics}
\usepackage{dsfont}
\usepackage{appendix}
\usepackage{dsfont}
\usepackage{arydshln}
\usepackage{mathtools,enumerate}
\usepackage{caption}
\usepackage{authblk}
\usepackage{stmaryrd}
\usepackage[a4paper]{geometry}
\usepackage{comment}

\newcommand{\Z}{\mathbb{Z}}
\newcommand{\R}{\mathbb{R}}
\newcommand{\C}{\mathbb{C}}
\newcommand{\N}{\mathbb{N}}

\newcommand{\per}{\mathrm{per}}
\newcommand{\ab}{\mathrm{ab}}
\newcommand{\uni}{\mathrm{uni}}
\newcommand{\T}{\mathbb{T}}

\newcommand{\EE}{\mathrm{e}}
\usepackage{etoolbox} \apptocmd{\thebibliography}{\setlength{\itemsep}{1pt}  \setlength{\parskip}{0pt} }{}{}

\title{A sandwich theorem for the spectrum of graph coverings}
\author[]{Charles Bordenave, Wenbo Li and Joe Thomas}

\date{}
\begin{document}
\maketitle

\newtheorem{theorem}{Theorem}[section]
\newtheorem{proposition}[theorem]{Proposition}
\newtheorem{corollary}[theorem]{Corollary}
\newtheorem{lemma}[theorem]{Lemma}
\newtheorem{definition}[theorem]{Definition}
\newtheorem{remark}[theorem]{Remark}
\newtheorem{question}{Question}[theorem]
\newtheorem{conj}[theorem]{Conjecture}
\newtheorem{eg}[theorem]{Example}


{
\newcommand{\JamesKirk}{12 mm}
\newcommand{\Spock}{6 mm}
\newcommand{\CaptainPike}[3]{
  \vcenter{\vbox to #2{
    \vfil
    \hbox to #1{\hfil$ #3$\hfil}
    \vfil
  }}
}
\newcommand{\JamesKirkrow}[4]{
  \CaptainPike{\JamesKirk}{\JamesKirk}{#1} &
  \CaptainPike{\Spock}{\JamesKirk}{#2} &
  \CaptainPike{\Spock}{\JamesKirk}{#3} &
  \CaptainPike{\JamesKirk}{\JamesKirk}{#4}
}
\newcommand{\Spockrow}[4]{
  \CaptainPike{\JamesKirk}{\Spock}{#1} &
  \CaptainPike{\Spock}{\Spock}{#2} &
  \CaptainPike{\Spock}{\Spock}{#3} &
  \CaptainPike{\JamesKirk}{\Spock}{#4}
}
\newcommand{\USSEnterprise}[4]{
  \begingroup
  \setlength{\arraycolsep}{0pt}
  \renewcommand{\arraystretch}{0}
  \setlength{\dashlinedash}{1.2pt}
  \setlength{\dashlinegap}{1.2pt}
  \left(
  \begin{array}{@{}c:c:c:c@{}}
    \JamesKirkrow#1 \\ \hdashline
    \Spockrow#2 \\ \hdashline
    \Spockrow#3 \\ \hdashline
    \JamesKirkrow#4
  \end{array}
  \right)
  \endgroup
}

{
\newcommand{\Picard}{11 mm}
\newcommand{\Riker}{5.5 mm}
\newcommand{\Data}[3]{
  \vcenter{\vbox to #2{
    \vfil
    \hbox to #1{\hfil$ #3$\hfil}
    \vfil
  }}
}
\newcommand{\Picardrow}[4]{
  \Data{\Picard}{\Picard}{#1} &
  \Data{\Riker}{\Picard}{#2} &
  \Data{\Riker}{\Picard}{#3} &
  \Data{\Picard}{\Picard}{#4}
}
\newcommand{\Rikerrow}[4]{
  \Data{\Picard}{\Riker}{#1} &
  \Data{\Riker}{\Riker}{#2} &
  \Data{\Riker}{\Riker}{#3} &
  \Data{\Picard}{\Riker}{#4}
}
\newcommand{\USSEnterpriseD}[4]{
  \begingroup
  \setlength{\arraycolsep}{0pt}
  \renewcommand{\arraystretch}{0}
  \setlength{\dashlinedash}{1.2pt}
  \setlength{\dashlinegap}{1.2pt}
  \left(
  \begin{array}{@{}c:c:c:c@{}}
    \Picardrow#1 \\ \hdashline
    \Rikerrow#2 \\ \hdashline
    \Rikerrow#3 \\ \hdashline
    \Picardrow#4
  \end{array}
  \right)
  \endgroup
}
\begin{abstract}
    We show that all normal covers of a fixed finite graph \(G\) sitting between the universal and maximal abelian covers share the same atomic part of their spectral measures. We give an explicit double formula for the mass of atoms which generalises and unifies known results on universal covering trees and maximal abelian covers. It also proves an extension of a conjecture in physics literature on hyperbolic lattices. Moreover, we also establish the logarithmic H\"older regularity of the continuous part of the spectral measures. Our proof combines techniques from von Neumann algebras, matching theory and the monotone labelling method.
    
    In an appendix, we prove a converse generalised Gallai-Edmonds structural theorem which might be of independent interest. 
\end{abstract}

\section{Introduction }\label{sec:intro }

\paragraph{Motivation.} In mathematical physics, the spectral theory of infinite graphs with certain symmetries, especially those that allow a free and co-compact group action has drawn a lot of attention over the years, see for example \cite{10.1007/3-540-16777-3_74,MoharWoess1989,Sunada1994,zbMATH07207199,DresselhausDresselhausJorio2008}. When such a group action exists, the infinite graph is actually a normal cover of a finite graph $G$, with the deck group $\Gamma$ acting on it.  Denote by $G_\Gamma$ the corresponding covering graph. Assume both \(G\) and \(G_\Gamma\) are connected. If the deck group is $\Gamma \cong \pi_1(G)$, the fundamental group of the graph, then $\Gamma$ is free and $G_{\pi_1(G)} \cong G^{\uni}$ is the universal covering tree of $G$. If the deck group is $\Gamma \cong H_1 (G;\mathbb Z)$
then $G_\Gamma \cong G^{\mathrm{ab}}$ is the maximal abelian cover of $G$. When $\Gamma$ is a Fuchsian group, then $G_\Gamma$ is referred to as a hyperbolic lattice in the physics literature.

The goal of this work is to establish a surprising comparison theorem  between the spectral measures (called density of states in physics) of normal coverings of a given finite graph for different choices of deck groups $\Gamma$ whose first Betti number is maximal. Under this topological assumption, our result notably proves a conjecture made in the physics literature \cite{bzduvsek2022flat} of equality of atoms for the spectral measures (called flat bands in physics) of the corresponding normal coverings. When this assumption fails, we shall also provide counter-examples of their conjecture. Importantly, beyond atoms, our comparison theorem also allows one to prove logarithmic H\"older regularity of the continuous part of the spectral measures.

\paragraph{Spectral measure of Schr\"{o}dinger operators of graph covers.} 

Fix a finite graph $G=(V,E)$ (possibly with multiple edges and loops) and choose an arbitrary orientation on each of the edges to get a collection of oriented edges $E_+$. The reverse edge of $e \in E_+$ is denoted by $\bar e$. A Schr\"{o}dinger operator $\mathcal{H}$ is a linear operator on $\ell^2(V)$ so that 
    $$\label{eq:Schrodinger_op}
    \mathcal{H}=\sum_{e\in E_+}w_eA_e+w_{\bar{e}}A_{\bar{e}}+\mathcal{V},$$
where  for $e$ joining the vertex $u$ to $v$, $A_e$ is the matrix with $1$ in the $(u,v)$ entry and $0$ elsewhere, $w_{\bar{e}}=\bar{w_e} \ne 0$, 
and $\mathcal{V}$ is a diagonal matrix with real entries $(\mathcal V_v)_{v \in V}$. Note that the operator $\mathcal H$ is self-adjoint. 

Let \(\pi:G_\Gamma\to G\) be a normal cover of the finite graph \(G\) with deck
group \(\Gamma\). Any Schr\"odinger operator \(\mathcal H\) on
\(G\) lifts to a periodic Schr\"odinger operator
\(\mathcal H_\Gamma\) on \(G_\Gamma\) that is invariant under the group action of \(\Gamma\). The spectral measure of \(G_\Gamma\) with respect to \(\mathcal H_\Gamma\)  is the probability measure defined by
\begin{equation}\label{eq:DOS}
\mu_{G_\Gamma}(B)
:=
\frac{1}{|V|}
\sum_{v\in V_G}
\left\langle
\mathbf 1_B(\mathcal H_\Gamma)\delta_{\widetilde v},
\delta_{\widetilde v}
\right\rangle,
\qquad B\subseteq\mathbb R\; \text{is a Borel set},
\end{equation}
where for each \(v\in V = V(G)\), we choose its lift \(\widetilde v\in p^{-1}(v)\). Here \(\mathbf 1_B(\mathcal H_\Gamma)\) is defined by functional calculus, the orthogonal projection onto the spectral subspace of
\(\mathcal H_\Gamma\) corresponding to \(B\). The case where \(B=\{\theta\}\) for \(\theta\in \R\) is of particular interest in operator algebras and mathematical physics. When \(\mu_{G_\Gamma}(\{\theta\})>0\), \(\theta\) is an \(\ell^2\)-eigenvalue of \(\mathcal H_\Gamma\) and \(\mu_{G_\Gamma}(\{\theta\})\) is called its (renormalised) \textit{multiplicity}. 
\paragraph{Atomic multiplicity. }  For certain \(\Gamma\), the location and multiplicity of \(\ell^2\)-eigenvalues can be computed explicitly. When \(G_\Gamma\) is the universal cover, the exact formula for multiplicities is given by Salez \cite[Theorem 2]{salez2020spectral} in the language of unimodular graphs and by Banks, Garza-Vargas and Mukherjee {\cite[Theorem 3.1]{banks2022point}} in terms of Aomoto
sets on \(G\). When \(G_\Gamma\) is the maximal abelian cover, Magee, Sabri and the second and third named authors \cite{LiMageeSabriThomas} gave a
matching polynomial criterion for the location of \(\ell^2\)-eigenvalues. This was used to solve a conjecture of Higuchi and Nomura \cite{higuchi2009spectral}, proving the absence of eigenvalues for maximal abelian covers of regular multi-graphs. Building on these results, Spier \cite[Theorems 7 and 8]{spier2025eigenvalues} proved that the universal and the maximal abelian cover of a finite multi-graph always have the same eigenvalues, thereby extending the result in \cite{LiMageeSabriThomas}. 

The multiplicity of eigenvalues can sometimes also be studied with the absence of an exact formula. For example, for \(\Gamma\) satisfying the strong Atiyah conjecture, the multiplicity is always a rational number. Furthermore, when \(\Gamma\) is torsion-free there will be no \(\ell^2\)-eigenvalues when \(G\) is a bouquet. This applies for example to right-angled Artin and Coxeter groups and to all one-relator groups as shown by Linnell, Okun and Schick   \cite{zbMATH06050472} and  S{\'a}nchez-Peralta \cite{zbMATH08052380} respectively. More recently, the first named author extended the monotone labelling 
method \cite{bordenave2026logarithmic} developed by Sen, Vir\'ag and himself \cite{Bo.Se.Vi2017} to exclude \(\ell^2\)-eigenvalues on Cayley graphs on \(\Gamma\) satisfying an indicability condition where \(\Gamma\) can even have torsion. The same method has been subsequently applied by Nachmias \cite{nachmias2026regularhyperbolictilingsell2} to exclude \(\ell^2\)-eigenvalues on regular-tilings of the hyperbolic plane, corresponding to the surface group. 

In the physics literature, there has been a growing interest on {non-abelian} coverings, particularly when \(G_\Gamma\) has a Fuchsian deck group where these are referred to as hyperbolic lattices. Tight-binding models on such graphs have been realised in experiments, and the non-abelian symmetries have led to the development of non-abelian Bloch theory \cite{kollar2019hyperbolic,maciejko2022automorphic,Lenggenhager2023Nonabelian,Boettcher2022}.

\paragraph{Main result. }
Let \(\pi_1: X_1\rightarrow X\) and \(\pi_2: X_2\rightarrow X\) be two covering maps. We say \(\pi_1\) factorises through \(X_2\) if there exists a covering map \(\pi': X_1\rightarrow X_2\) such that 
\(\pi_2 \circ \pi'=\pi_1\).
The universal covering map  $G^{\uni} \to G$ factorises through $G_\Gamma$ for any normal cover \(\pi_\Gamma : G_\Gamma\rightarrow G\). Also, $\pi_\Gamma$ always factorizes through a covering graph \(G_{\Gamma^\ab} \) with deck group \(\Gamma^\ab=\Gamma/[\Gamma,\Gamma]\). 

In the physics literature  \cite{bzduvsek2022flat}, Bzdu{\v{s}}ek and Maciejko have observed that for certain \(G_\Gamma\), the atomic part of the spectral measure of \(G_\Gamma\) is the same as that of \(G_{\Gamma^\ab}\). They conjectured that this was true for all Fuchsian groups \(\Gamma\). This, however, fails in general without further assumptions, as we shall see Section \ref{Section:eg}. On the other hand,   we reveal a rigidity result of the atomic part of spectral measure and a logarithmic regularity control for the continuous spectrum.

\begin{theorem}\label{conj}
     Let \(G_\Gamma\) be a connected normal cover of a finite connected graph \(G = (V,E)\) with deck group \(\Gamma\). If the covering map factorises through \(G^\ab\), then there exists constants \(c,\alpha>0\) depending on \(G\), \(\Gamma\) and the prescribed \(\mathcal H\), such that for any small enough closed interval \(I\subset \mathbb R\),  
\[
\left|\mu_{G_\Gamma}(I)
-
\mu_{G^{\mathrm{ab}}}(I)\right|\leq \frac{c}{\log^\alpha (1 / |I|) }.
\]
Moreover, for any \(\theta \in \mathbb R\), we have the following double formula for possible \(\ell^2\)-eigenvalues
\begin{equation}
\label{eq:sandwich} \mu_{G_\Gamma}(\{\theta\})=\frac{1}{|V|}\max_{S\in\mathcal A_\theta(G)}
\left(\operatorname{cc}S-|\partial S|\right)= \frac{1}{|V|}\min_{\gamma\in\mathcal C(G)}
\operatorname{mult}_\theta(G\setminus\gamma),\end{equation}
where \(\mathcal A_\theta (G)\) and \(\mathcal C(G)\) are, respectively, the collection of Aomoto sets and \(2\)-regular subgraphs of \(G\), and \(\operatorname{mult}_\theta(G\setminus\gamma)\) is the multiplicity of \(\theta\) as zero of the generalised matching polynomial of \(G\setminus\gamma\).
\end{theorem}

\begin{figure}[ht]
\centering
\includegraphics[width=0.35\linewidth,scale=0.5]{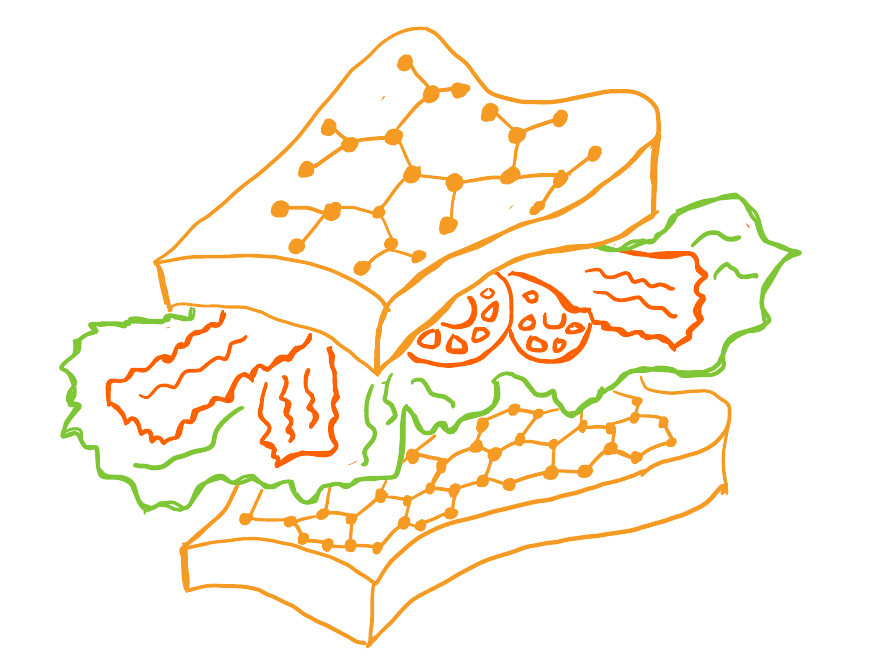}
\caption{A “fantastic” drawing of a BLT sandwich from the authors. The top and bottom slices illustrate the universal (the \(3\)-regular tree) and maximal abelian cover (the honeycomb lattice) of a theta graph, respectively.}
\end{figure}

We shall recall the definitions of Aomoto sets in Section \ref{sec:spier} and of the matching polynomial in Section \ref{sec:comb}. The first formula in \eqref{eq:sandwich} is the formula for $\mu_{G^{\uni}}(\{\theta\})$ as given  in \cite{banks2022point}.  The second formula is a formula for $\mu_{G^{\mathrm{ab}}}(\{\theta\})$ that we will obtain as a refinement of \cite{LiMageeSabriThomas}. Since $G^{\uni}$ factorises through $G_{\Gamma}$ which factorises through $G^{\mathrm{ab}}$, Theorem \ref{conj} can thus be interpreted as a sandwich theorem for $\mu_{G_\Gamma}$ between $\mu_{G^{\uni}}$ and $\mu_{G^{\mathrm{ab}}}$.

 Note that the continuous part of $\mu_{G^{\uni}}$ and $\mu_{G^{\mathrm{ab}}}$ are in fact absolutely continuous. For the familiar reader, we mention that when the above formula of atoms is established for \(G^{\uni}\) and \(G^{\mathrm{ab}}\), it holds already for \(\Gamma\) satisfying additional group-theoretic assumptions, see \cite{JaikinZapirain2021}. These conditions are satisfied for certain groups like the surface group. However, they do not hold for the generality of \(\Gamma\) considered here and further, they give no control of the continuous part of the spectrum.

We recall also that the ranks of $\Gamma^\ab$ and $H_1(G ; \mathbb Z)$ are by definition, respectively, the first Betti numbers $b_1(\Gamma)$ and $b_1(G)$.   For connected \(G\) and \(G_\Gamma\), we have \(b_1(G)=|E(G)|-|V(G)|+1\) and  $b_1 (\Gamma) \leq b_1(G)$. The factorisation of coverings assumption is equivalent to $b_1(\Gamma) = b_1(G)$. We could thus summarise Theorem \ref{conj} as follows: The spectral measure of all connected normal covers of a finite connected graph with deck groups of maximal first Betti number have the same atomic part (and a logarithmic Hölder regular continuous part).  When this maximality fails, we provide a no-go result in Proposition \ref{prop:sharp}, for example, when the non-abelian deck group \(\Gamma\) is a one-relator group or is torsion free.

The matching polynomial of the empty graph is the constant $1$ by definition. Hence, Theorem \ref{conj} implies that for any normal covers factorising through the maximal abelian cover of a graph \(G\) that contains a degree $2$ subgraph containing all of its vertices (a $2$-factor), the spectral measure is log-Hölder continuous. The same conclusion holds if $G$ is a regular graph since it is shown in  \cite{LiMageeSabriThomas} that for any \(\theta\in \R\), one can always find a \(2\)-regular subgraph \(\gamma\) such that \(\operatorname{mult}_\theta(G\setminus\gamma)=0\).

We finally note that the main results in \cite{bordenave2026logarithmic} depend on a homomorphism \(\Gamma \rightarrow \Z\), and the first author asked if more can be said when \(\Z\) is replaced by a general \(\Z^d\). Theorem \ref{conj} is one possible illustration of this phenomenon.

\paragraph{Overview of the proof and organisation of the paper.} The proof of Theorem \ref{conj} combines some techniques from topological graph theory, von Neumann algebras and algebraic matching theory.

Section \ref{sec:op_algs} introduces our setting from an operator algebra viewpoint. We first realise the covering graph $G_\Gamma$ as a voltage graph and write the Schrödinger operator $\mathcal H_\Gamma$ as an element of a finite dimensional extension over the group algebra $\mathbb C \Gamma$. The spectral projectors of $\mathcal H_\Gamma$, $\mathrm{1}_B(\mathcal H_\Gamma)$ for $B \subset \mathbb R$ Borel, are then elements of the associated finite von Neumann algebra and $\mu_{G_\Gamma} (B)$ is expressed as the von Neumann dimension of the corresponding invariant linear subspace. The basic properties of the von Neumann dimension will play a crucial role in the sequel. 

Section \ref{sec:comb} surveys some properties of the weighted matching polynomial of a finite graph including recursive identities due to Heilmann and Lieb. We then introduce the generalised Gallai-Edmonds decomposition. This theory connects structural properties of a graph with the zeroes of its matching polynomial and their multiplicities.   It was a key input in the analysis of Spier \cite{spier2025eigenvalues}.  In Appendix \ref{appendix}, we also present a converse Gallai-Edmonds theorem which could be of independent interest (but which is not used in the proof of Theorem \ref{conj}).

The actual proof of Theorem \ref{conj} starts with Section \ref{sec:abelian}. With the notation of equation \eqref{eq:sandwich}, we prove  in Theorem \ref{prop:multicriterion} the identity  $$\mu_{G^\ab} (\{\theta\})  = \frac{1}{|V|}\min_{\gamma\in\mathcal C(G)}
\operatorname{mult}_\theta(G\setminus\gamma).$$
The proof extends the previous work \cite{LiMageeSabriThomas} where it was shown that $\mu_{G^\ab} (\{\theta\}) \ne 0$ if and only if $\min_{\gamma\in\mathcal C(G)}
\operatorname{mult}_\theta(G\setminus\gamma) \ne 0$. To prove this identity, we first use harmonic analysis (or Floquet theory) to express $\mu_{G^\ab} (\{\theta\})$ in terms of the dimension of kernels of a family of Schrödinger operators on $G$. We then exploit an identity from \cite{LiMageeSabriThomas} which expresses the characteristic polynomial of  Schrödinger operators on $G$ as a weighted sum of the matching polynomials of $G\setminus \gamma$, $\gamma \in \mathcal C(G)$ (and which can be traced back to Godsil-Gutman \cite{zbMATH03722694}).

A key combinatorial idea underpinning our results is the passage of spectral problems to the analysis of the matching structure of an associated bipartite graph built with the Gallai-Edmonds decomposition in mind. This crucially provides a toolkit to study the $G_\Gamma$ of interest, replacing the Floquet-Bloch theory that is unavailable because $\Gamma$ is non-abelian. The main motif is that the subgraphs of the graph responsible for $\theta\in\mathbb{R}$ being a matching polynomial root ($\theta$-critical components) persist in being responsible for $\theta$ being an atom for $G_\Gamma$. This is revealed in the block decomposition of $\mathcal{H}_\Gamma$ when we analyse it through the operator algebraic formulation. The associated bipartite graph that we study treats each of these subgraphs as a single vertex and allows us to study how they sit within the remainder of the graph.

Such a bipartite graph appears as \(G_\theta\) in Lemma \ref{lem: kappa_lb} of Section \ref{sec:spier}, as \(W\) in the proof of Lemma \ref{lemma:Halltype} in Section \ref{sec:monotone}, and as \((X\sqcup Y,E)\) in the setting of Theorem \ref{conversedGallai-Edmonds} in the appendix. The method is motivated by both the result and the proof of a “converse” version of the classic Gallai-Edmonds theorem,  which can be traced back to at least Lov\'asz and Plummer \cite[Theorem 3.2.3]{lovasz2009matching}. We generalise it by an essential use of the multi-edge recursion relation of matching polynomials, see equation \eqref{eq:multiedgerecursion}.

 This idea in particular leads to a direct and shorter proof of Spier's theorem \cite{spier2025eigenvalues} and the establishment of its multiplicity refinement in Section \ref{sec:spier}, see Theorem \ref{thm:lowerbound}. Our aim is to establish that for any connected normal cover $G_\Gamma$ (not necessarily sitting between \(G^\ab\) and \(G^\uni\)) we have the lower bound:
$$
\mu_{G_\Gamma}(\{\theta\}) \geq \frac{1}{|V|}\max_{S\in\mathcal A_\theta(G)}
\left(\operatorname{cc}S-|\partial S|\right) \geq  \frac{1}{|V|}\min_{\gamma\in\mathcal C(G)}
\operatorname{mult}_\theta(G\setminus\gamma).
$$
  As mentioned, the second inequality builds on Spier and relies on the generalised Gallai-Edmonds decomposition theorem. The first inequality ultimately comes as a consequence of the rank nullity theorem in von Neumann algebras. 

The proof of Theorem \ref{conj} is concluded in Section \ref{sec:monotone}. We prove that for any closed interval $I$
$$
\mu_{G_\Gamma}(I)
\leq 
\sum_{\text{\(\theta \in I\) is an \(\ell^2\)-eigenvalue of \(G^\ab\)}}\mu_{G^{\mathrm{ab}}}(\{\theta\})+ \frac{c}{ \log^{\alpha} (1/|I|)},
$$
for some $c,\alpha$ when $|I|$ is small enough. The proof of this inequality depends on whether or not $\theta$ is such that $$\min_{\varnothing\neq\gamma\in\mathcal C(G)}
\operatorname{mult}_\theta(G\setminus\gamma) > 0 \qquad \text{and}\qquad 
\operatorname{mult}_\theta(G)= 0.$$ If this holds, we shall prove in Theorem \ref{thm:surprisethm} a much stronger gap phenomenon: $\theta$ is not in the spectrum
of $\mathcal H_{G_\Gamma}$ for any normal cover \(G_\Gamma\) and thus $\mu_{G_\Gamma}(\theta-\epsilon,\theta+\epsilon) = 0$ for all $\epsilon$ small enough.   This might hold intrinsic  interest in spectral graph theory, for example with connections to the full spectrum conjecture of maximal abelian covers \cite{sunada2008discrete}. 

In any other case, we use the monotone labelling technique developed in \cite{Bo.Se.Vi2017,bordenave2026logarithmic} together with a recursion on the size of the edge set of $G$. We present here a new application of the monotone labelling technique which avoids the use of randomness as is done in \cite{Bo.Se.Vi2017,bordenave2026logarithmic}. This technical novelty may find applications elsewhere.

  We conclude with Proposition \ref{prop:sharp} in Section \ref{Section:eg} to complement Theorem \ref{conj}. For any finitely generated group \(\Gamma\) that has a torsion free element in the commutator subgroup \([\Gamma,\Gamma]\), we present a finite graph \(G\)  such that $\mu_{G_\Gamma}$ and  $\mu_{G_{\Gamma^{\mathrm{ab}}}}$ have different sets of atoms while \(G_\Gamma\) and \(G_{\Gamma^{\mathrm{ab}}}\) are both connected. This in particular gives a counter-example of the conjecture made in \cite{bzduvsek2022flat} where \(\Gamma\) is a Fushian group. In this proposition,  $G$ is  constructed so that $b_1(\Gamma)=b_1(G)-1$ contrary to the assumption in Theorem \ref{conj}. A monotone labelling type of argument will be used to exclude certain eigenvalues on \(G_{\Gamma}\) while a cut vertex provides a finitely supported eigenfunction on \(G_{\Gamma^{\mathrm{ab}}}\).

\paragraph{Notation}

Throughout, we will be considering multi-graphs $G=(V,E)$, which means we allow for {multi-edges} between vertices as well as {self-loops} based at a vertex. We will refer to them simply as graphs. We denote by $\operatorname{cc} G$ the number of connected components of $G$. We denote by $\vec{E}(G)$ (or just $\vec{E}$ if the context is clear) the collection of edges of $G$ with both possible orientations. If $e$ is an oriented edge of $G$, then we let $o(e)$ denote its origin, $t(e)$ denote its terminus and \(\bar{e}\) the same edge in a different orientation. For vertices $u,v\in V$, we write $u\sim v$ if there exists an edge in $E$ connecting them. For \(S\subset G\), denote by \(\partial S=\partial S(G)\) its \textbf{vertex boundary} in \(G\), i.e., 
\[\partial S=\{v\in V\setminus S: \exists\, u\in S,  v\sim u\}.\]

We also use the following conventions:
    \begin{itemize}
        \item If $G_1$ is a subgraph of $G$, then $G\setminus G_1$ is the graph $G$ with the edges and vertices of $G_1$ removed from it along with any edges that connect vertices in $G_1$ to vertices outside of $G_1$.
        \item If $S\subset V$, then $G\setminus S$ is the graph $G$ with the induced subgraph on $S$ removed.
        \item If $F\subset E$, then $G\setminus F$ is the graph $G$ with just the edges in $F$ removed not the vertices that consist of their endpoints.
    \end{itemize} 

\paragraph{Acknowledgments.} The authors thank Jorge Garza-Vargas and Oriol Sol\'e Pi for related discussions at an early stage of this project. 

C.B. has started this work during his stay at the Institute for Advanced Study. He acknowledges support from the James D. Wolfensohn Fund and from French ANR grant Tagada (ANR-25-CE40-5672). W.L. has started this work during his visits of J.T. in Durham University and of C.B. in Institute for Advanced Study and he thanks them and Shouda Wang for the hospitality. His research is supported by the Scientific Research Innovation Capability Support Project for Young Faculty (SRICSPYF-ZY2025160) and the National Natural Science Foundation of China (Grant No. 123B2013).  J.T. has received funding from the Leverhulme Trust through a Leverhulme Early Career Fellowship (Grant No. ECF-2024-440). Part of this work was conducted while J.T. was visiting the Simons Institute for the Theory of Computing. 

\paragraph{Statement on AI.} The authors thank the USS Enterprise and her crew for inspiring human exploration of the unknown. The contents of this paper are the authors' own. LLMs were not used at any stage in the writing of this work other than for bibliographical searching. In particular, they brought the authors' attention to the physics literature \cite{bzduvsek2022flat}.

\section{Background on voltage graphs and operator algebras}

\label{sec:op_algs}

\subsection{Voltage graphs}

We begin by recasting the covers $G_\Gamma$ and their Schr\"{o}dinger operators into a spectrally equivalent formulation in the language of operator algebras. The easiest way to describe this is through voltage graphs which form a special case of what are termed as quasi-transitive graphs in \cite{bordenave2026logarithmic}. We refer to \cite[Chapter 2]{Gr.Tu2001} for further background on voltage graphs.

\begin{definition}[Derived voltage graph]
    Let \(G\) be a multi-graph and  \(\Gamma\) a group. Assign each directed edge \(e\in \vec E(G)\) an element \(g_e\in \Gamma\) such that \(g_e=g_{\bar e}^{-1}\). $G$, together with this assignment, is called a voltage graph.  The derived voltage graph \(G(\Gamma)\) is constructed as follows:
        \begin{enumerate}[(i)]
            \item the vertex set is \(V(G)\times \Gamma\),
            \item there is a (directed) edge between \((u,x),(v,y)\) whenever there is a directed edge \(e\) from \(u\) to \(v\) in \(G\) and \(y=xg_e\) in \(\Gamma\). 
        \end{enumerate} 
        Note that for any induced subgraph \(H\) of \(G\), the restriction of the voltage assignment on \(H\) also gives a derived voltage graph.
\end{definition}

By construction, $\Gamma$ has a natural free action on $G(\Gamma)$ given by the left regular representation on the $\Gamma$ factor. That is, $\Gamma$ acts on vertices of $G(\Gamma)$ by $g\cdot (u,x)=(u,gx)$ and this induces an action on the edges. 

\begin{remark}
    Note that \(G(\Gamma)\) is connected if and only if there is a vertex $v\in V(G)$ such that words in $\Gamma$ obtained by tracing out the closed walks based at $v$ generate $\Gamma$. Moreover, if \(G\) is a bouquet graph, i.e., a graph with a single vertex and several self-loops, then the derived voltage graph obtained  is the Cayley graph $\mathrm{Cay}(\Gamma,S)$ where $S$ is the collection of edge assignments on $G$.
\end{remark}

\begin{eg}\label{def of Gper}
    Let $G$ be a finite connected graph and $E_+$ an assignment of direction to each edge in $E(G)$. Let \(\Gamma=\Z^{E_+}\) be the free abelian group generating by \(E_+\). Define $E_f\in \Z^{E_+}$ to be the vector with $0$ in all but the $f$-th entry where it is $-1$ if $f\in E_+$ and $1$ if $\bar{f}\in E_+$. For any \(e\in \vec{E}(G)\), assign \(g_e=E_e\). We denote the corresponding derived voltage graph as \(G^{\per}.\) In \cite[Proposition 2.3]{LiMageeSabriThomas} it is proven that every connected component of $G^\per$ is isomorphic to the maximal abelian cover $G^\ab$. 
\end{eg}

Let $G(\Gamma)$ be a derived voltage graph constructed from a finite graph $G$ and a group $\Gamma$. Then, $G(\Gamma)$ gives a (possibly disconnected) normal covering of $G$ with deck group given by $\Gamma$ (acting by left regular representation as above). The fibre of a vertex $v\in V(G)$ is $\{v\}\times\Gamma$.

It is also true that every connected component of $G(\Gamma)$ is a normal covering of $G$. The deck group of this covering is given by the subgroup of $\Gamma$ generated by the words obtained by reading off the edge labels when traversing all of the closed walks based at any fixed vertex $v\in V(G)$.

\begin{eg}[Example \ref{def of Gper} continued] 
By construction, $G^\mathrm{per}$ has deck group $\Z^{E_+}$ whereas $G^\ab$ has deck group $H_1(G;\Z)$ which is isomorphic to $\Z^{b_1(G)}$, where $b_1(G)=\operatorname{rank} H_1(G;\Z)$ is the number of edges of $G$ that are not contained in a fixed spanning tree of $G$.
\end{eg}

On the converse, any connected normal covering graph of a connected graph \(G\) can be achieved as a derived voltage graph. To see this, let \(\pi:G_\Gamma\to G\) be a normal cover with deck group
\(\Gamma\) acting by
\[\Gamma \times V({G_\Gamma}) \rightarrow V({G_{\Gamma}}),\qquad (g,x)\mapsto g\cdot x.\] 

Fix a spanning tree \(T\subseteq G\) and choose a lift \(\widetilde T\) of it in
\(G_\Gamma\). For each \(v\in V(G)\), the choice of $\tilde{T}$ gives a chosen lift
\(\widetilde v_0\in \pi^{-1}(v)\) residing in $\tilde{T}$. For any \(\widetilde v \in V({G_\Gamma})\) with \(\pi(\widetilde v)=v\), there exists a unique \(g\in \Gamma\) such that \(\widetilde v=g\cdot \widetilde v_0\). Here existence and uniqueness is guaranteed since \(\pi\) is a normal cover and the covering graph is connected. We identify
\[
\label{eq: voltage-cover_corr}
V({G_\Gamma})
= V(G) \times \Gamma,
\qquad
g\cdot \tilde{v}_0\leftrightarrow(v, g).
\]

Let \(\vec E(G)\) be the set of oriented edges of \(G\), and let
\(\vec E(T)\subseteq \vec E(G)\) be the oriented edges belonging to \(T\). For
each \(e\in \vec E(G) \), with origin \(u\) and terminus
\(v\), there is a unique element \(g_e\in\Gamma\) such that the lift of
\(e\) starting at \(\tilde{u}_0\) ends at
\(g_e\cdot\tilde{v}_0\). By definition of the deck group,  for any \(g\in \Gamma\), the lift of
\(e\) starting at \(g\cdot \tilde{u}_0\) ends at
\(\bigl( gg_e\bigr) \cdot\tilde{v}_0\). Note that for $e\in \vec{E}(T)$, we have $g_e=\EE$, the unit of $\Gamma$.

This assignment gives $G$ a voltage graph structure and the derived voltage graph $G(\Gamma)$ is isomorphic to $G_\Gamma$ under the correspondence of the vertex set (and its induced mapping of the edges) given in \eqref{eq: voltage-cover_corr} (see also \cite[Theorem 2.2.2]{Gr.Tu2001}).

The following lemma will be used in Section \ref{sec:monotone}.
\begin{lemma}\label{lemma:Hughesfree}
    Assume that \(G\) is connected. Let $G_\Gamma$ be a normal cover of $G$ that factorises through $G^\ab$. For any connected subgraph \(H\) of \(G\), the restriction of \(\pi_\Gamma\) on each connected component of \(\pi_\Gamma^{-1}(H)\) factorises through the maximal abelian cover of \(H\).
\end{lemma}

\begin{proof}
    Fix a spanning tree $T_H$ of $H$ and extend it to a spanning tree $T$ of $G$. Now we can assign an arbitrary orientation $E^+(G)$ to the edges of $G$ which gives orientations $E^+(H)$, $E^+(T_H)$ and $E^+(T)$ to the edges of $H, T_H$ and $T$ respectively. 
    
    The maximal abelian covering $\pi_{\ab}:G^\ab\to G$ has deck transformation group $H_1(G;\Z)$ which is a free abelian group of rank $|E^+(G)\setminus E^+(T)|$. By the construction preceding the lemma, we can realise $G^\ab$ as a voltage graph of $G$ with edge labels in $H_1(G;\Z)$. For edges in $E(T)$, the labels are the identity and, since the maximal abelian cover is by definition connected, the labels $(g_e)_{e\in E^+(G)\setminus E^+(T)}$ form a set of independent generators for $H_1(G;\Z)$.

    We can naturally see $H_1(H;\Z)$ as sitting inside $H_1(G;\Z)$ generated by $(g_e)_{e\in E^+(H)\setminus E^+(T_H)}$. In fact, with this realisation, given a set of transversals $(a_i)$ for $H_1(G;\Z)/H_1(H;\Z)$, we have that $(a_ig_e)_{e\in E^+(H)\setminus E^+(T_H)}$ generates another copy of $H_1(H;\Z)$ in $H_1(G;\Z)$ for each $i$. 

    Now, the derived voltage graph construction from $H$ using any of these bases as a voltage assignment to the edges (identity on the edges in $E(T_H)$) gives a copy of the maximal abelian cover of $H$ that is naturally sitting inside the derived voltage graph of $G^\ab$ as the induced subgraph on $V(H)\times \langle a_i g_e\rangle_{e\in E^+(H)\setminus E^+(T_H)}$. In fact, each such copy is a connected component in the preimage $\pi_\ab^{-1}(H)$ since this preimage is precisely the induced subgraph on
        $$V(H)\times H_1(G;\Z) = \bigsqcup_{i} V(H) \times \langle a_ig_e\rangle_{e\in E^+(H)\setminus E^+(T_H)}.$$

    Since $G_\Gamma$ factors through $G^\ab$, we have $\pi_\Gamma = \pi_\ab \circ \pi_{\Gamma,\ab}$ where $\pi_{\Gamma,\ab}: G_\Gamma\to G^\ab$ is a normal covering. But then given any connected component $C$ in the pre-image of $\pi_\Gamma^{-1}(H)$, we have that $\pi_\Gamma(C)$ has connected image in $\pi_\ab^{-1}(H)$ and so must be contained in the induced subgraph on $V(H) \times \langle a_ig_e\rangle_{e\in E^+(H)\setminus E^+(T_H)}$ for some $i$, but this means it factors through $H^\ab$.
\end{proof}

Let \(\rho_\Gamma:\Gamma\to B(\ell^2(\Gamma))\) be the right regular representation and \(\C\Gamma\) be the associated group algebra. Fix any Schr\"odinger operator \(\mathcal H\) on \(G\) with \(A_0\) its restriction on the spanning tree
\(T\). Using the correspondence of $G_\Gamma$ with its voltage graph $G(\Gamma)$ constructed by a spanning tree $T$, the lifted Schr\"odinger operator \(\mathcal H_\Gamma\) to $G_\Gamma$ corresponds to the operator 
\[
\label{eq:lifted_schrodinger}
\mathcal H_\Gamma
=
A_0\otimes \mathrm{id}
+
\sum_{e\in \vec E(G)\setminus \vec E(T)}
w_eA_e\otimes \rho_\Gamma(g_e),
\]
acting on $\ell^2(V(G))\otimes \ell^2(\Gamma)$. We can also view $\mathcal{H}_\Gamma$ as an element of \(M_{V(G)}(\mathbb C\Gamma)\). 

 The action of the deck group on $G_\Gamma$ corresponded to the left regular representation action of $\Gamma$ on the voltage graph $G(\Gamma)$. In turn this induces an isometric action on $\ell^2(V(G))\otimes \ell^2(\Gamma)$ given by 
    \[g'\cdot \delta_v\otimes \delta_g =\delta_v\otimes \lambda_\Gamma(g')\delta_g = \delta_v\otimes \delta_{g'g}.\]
The fact that \(\mathcal H_\Gamma\) is invariant under the action of the deck transformation group is immediately seen from the fact that the left and right regular representations commute. \textbf{From now on,  we will always see the normal covers of a connected graph $G$ as a voltage graph and abuse the notation using \(G_\Gamma\) for both.}

\subsection{The von Neumann dimension and spectral measures}

In this section, we recall some facts of group von Neumann algebras and the related dimension theory. For more in this topic, see the standard monograph \cite{luck2002l2} of L\"uck. 

Recall that the group \textbf{von Neumann algebra} of a group \(\Gamma\) is defined by
\[\mathcal N\Gamma:=
\overline{\mathbb C\Gamma}^{\,\mathrm{weak}}
\subset B(\ell^2(\Gamma)),
\]
which is closed under bounded Borel functional calculus for
self-adjoint elements. Note that due to the von Neumann bicommutant theorem, \(\mathcal N\Gamma\) is exactly the collection of \(\Gamma\)-invariant bounded linear operators on \(\ell^2(\Gamma)\). 

Let \(\tau_\Gamma\) be the canonical trace on
\(\mathcal N\Gamma\) such that for any \(p\in \mathcal N\Gamma\), 
\(\tau_\Gamma(p):=\langle \delta_\EE,p\delta_\EE\rangle\) (where $\EE$ is the unit of $\Gamma$). For any \(n\in \N\), \(M_n(\mathcal N\Gamma)\) acts on \(\C^n\otimes \ell^2(\Gamma)\) in the standard way. For any \(P\in M_n(\mathcal N\Gamma)\), we use the matrix amplification of \(\tau_\Gamma\) and define
\[
\operatorname{Tr}_{M_n(\mathcal N\Gamma)}(P)
:=
\sum_{i=1}^n \tau_\Gamma(P_{ii}),
\qquad
P\in M_n(\mathcal N\Gamma).
\]

\begin{definition}
    A \textbf{\(\mathcal{N}\Gamma\)-module} \(M\) is a \(\Gamma\)-invariant linear subspace, i.e., a linear subspace of \(\C^n\otimes \ell^2(\Gamma)\) for some \(n\in \N\) that is invariant under the (amplified) left regular representation \(\mathrm{id}\otimes \lambda_\Gamma\). Let \(P=P^*=P^2\) be the projection of \(\C^n\otimes \ell^2(\Gamma)\) to the closure of \(M\). By \(\Gamma\)-invariance,  \(P\in M_n(\mathcal N\Gamma)\), and allows us to define the \textbf{von Neumann dimension} of $M$ by
\[
\dim_{\Gamma}M
:=
\operatorname{Tr}_{M_n(\mathcal N\Gamma)}(P).
\]
A homomorphism \(\varphi\) between \(\mathcal N\Gamma\)-modules \(M\) and \(N\) are linear maps between \(M\) and \(N\) that intertwines with the action of \(\Gamma\) on \(M\) and \(N\). The \textbf{rank} of \(\varphi\) is defined as \(\operatorname{rank}_\Gamma \varphi=\dim_{\Gamma} \operatorname{im}\varphi\).
\end{definition}

\begin{remark}
    In \cite{luck2002l2}, a \(\mathcal{N}\Gamma\)-module is called
a \(\mathcal{N}\Gamma\)-Hilbert module. We omit Hilbert here to avoid confusions, given these subspaces are not necessarily closed, thus a Hilbert space. It does not necessarily mean that the authors have opinions of D. Hilbert.
\end{remark}

The von Neumann dimension behaves like the standard dimension on finite dimensional vector spaces. For example, we have the following result, see, for example, \cite[Theorem 1.12]{luck2002l2}.

\begin{lemma}\label{lm:weakexact}
    If there is an injective homomorphism from a \(\mathcal N\Gamma\)-module \(M\) to \(N\), then
\[\dim_\Gamma M \leq \dim_\Gamma N.\]
In fact, for any sequence of \(\mathcal N\Gamma\)-modules
    \(M_1 \xrightarrow{\varphi} M_2 \xrightarrow{\psi} M_2\)
    where \(\varphi\) and \(\psi\) are injective and surjective, respectively, and \(\overline{\mathrm{im} \varphi}=\ker \psi\), 
    \[\dim_{\Gamma}M_2=\dim_{\Gamma}M_1+\dim_{\Gamma}M_3.\]
    
\end{lemma}

When \(\varphi\) is a \(\Gamma\)-invariant bounded operator from \(\C^n\otimes \ell^2(\Gamma)\) to \(\C^m\otimes \ell^2(\Gamma)\), or equivalently, a rectangular matrix \(M_{m\times n}(\mathcal N \Gamma)\), Lemma \ref{lm:weakexact} tells us that 
\begin{equation}\label{eq:vonNeumannnull}
    \dim_{\Gamma} \ker T+\operatorname{rank}_\Gamma T=n.
\end{equation}
 In fact, on such operators, the rank function satisfies the following properties akin to the finite dimensional case and can be deduced from Lemma \ref{lm:weakexact}, and see \cite[Lemma 2.27]{arizmendi2024universality} for a direct proof.
\begin{lemma}
\label{lem:rank-identities}
The rank defined above is a Sylvester rank function, i.e., for any element in  \(M_{m\times n}(\mathcal N \Gamma)\), the following holds.
    \begin{enumerate}[(i)]
\item  \(\operatorname{rank}_\Gamma(0)=0\), \(\operatorname{rank}_\Gamma(\mathrm{id})=1\) where \(\mathrm{id}=\rho_\Gamma(\EE)\).

\item  \(\operatorname{rank}_\Gamma(AB)\leq \min (\operatorname{rank}_\Gamma(A),\operatorname{rank}_\Gamma(B))\). 

\item \(\operatorname{rank}_\Gamma \begin{pmatrix}
A & 0 \\
0 & B
\end{pmatrix}=\operatorname{rank}_\Gamma A+ \operatorname{rank}_\Gamma B\).

\item \(\operatorname{rank}_\Gamma \begin{pmatrix}
A & B \\
0 & C
\end{pmatrix}\geq \operatorname{rank}_\Gamma A+ \operatorname{rank}_\Gamma C\).
    \end{enumerate}
\end{lemma}

Note that from the above properties of Sylvester rank function, one can show directly 
\begin{equation}\label{eq:rankineq}
    \mathrm{rank}_\Gamma (A+B)\leq \mathrm{rank}_\Gamma (A)+\mathrm{rank}_\Gamma (B)
\end{equation}
when \(A\) and \(B\) are of the same size, see, for example, \cite[Lemma 2.18]{arizmendi2024universality}.

\begin{lemma}[Cauchy interlacing]\label{lemma:Cauchy interlacing}
    Let \(A\in M_n(\mathcal N \Gamma)\) be a self-adjoint matrix on \([n]\times [n]\). For \(S\subset [n]\), let \(A_S\) be the restriction of \(A\) on \(S\times S\). Then for any closed interval \(I\subset \R\), 
    \[\mathrm{rank}_\Gamma \mathbf{1}_I(A)\leq  \mathrm{rank}_\Gamma \mathbf{1}_I(A_S)+ n-|S|.\]
\end{lemma}

\begin{proof} The lemma is a consequence  of  \cite[Lemma 3.2]{e3d42fc0-0605-3104-9334-59c1bd91bfa0}. We offer a proof here for completeness. Without loss of generality, suppose \(I=[-\epsilon,\epsilon]\) and fix \(I'=(-\epsilon-\delta,\epsilon+\delta)\) for \(\delta>0\).

We define a homomorphism from \(\mathrm{im} \mathbf{1}_I(A)\) to \(\mathrm{im} \mathbf{1}_{I'}(A_S)\oplus \C^{n-|S|}\otimes \ell^2(\Gamma)\) by mapping \(f=(f_1,f_2)\in \mathrm{im} \mathbf{1}_I(A)\) to \((\mathbf{1}_I(A_S)f_1,f_2)\), where \(f_1\) is the summand corresponding to \(S\) and \(f_2\) to \([n]/S\) and show that it is injective. If \((f_1,f_2)\) is in the kernel of the map, then we have \(\mathbf{1}_{I'}(A_S)f_1=0\) and \(f_2=0\). Since 
\(f\in \mathrm{im} \mathbf{1}_I(A)\), we have \[\lVert A_S f_1\rVert\leq \epsilon \lVert f\rVert=  \epsilon \lVert f_1\rVert. \]
Since \(\mathbf{1}_{I'}(A_S)f_1=0\), we have 
\[\lVert A_S \mathbf{1}_{(I')^c}(A_S) f_1 \rVert =\lVert A_S f_1\rVert\leq \epsilon \lVert f_1\rVert=\epsilon \lVert  \mathbf{1}_{(I')^c}(A_S) f_1 \rVert.\]
However, we also have 
\[(\epsilon+\delta)\lVert \mathbf{1}_{(I')^c}(A_S) f_1 \rVert \leq \lVert A_S \mathbf{1}_{(I')^c}(A_S) f_1 \rVert,\]
thus \[\mathbf{1}_{(I')^c}(A_S) f_1=0,\]
and hence \(f_1=0\). By Lemma \ref{lm:weakexact}, 
\[\mathrm{rank}_\Gamma \mathbf{1}_I(A)\leq  \mathrm{rank}_\Gamma \mathbf{1}_{I'}(A_S)+ n-|S|.\]
due to the arbitrariness of \(\delta\), the conclusion follows.
\end{proof}

\begin{definition}[The spectral measure for voltage graphs]
\label{def:dos}
Let $G  = (V,E)$ be a finite graph and \(G_\Gamma\) be a derived voltage graph with voltage assignment \((g_e)_e\). Let $$
    \mathcal{H}=\sum_{e\in \vec E}w_eA_e+\mathcal{V}\in M_{V}(\C)$$
be a Schr\"odinger operator on \(G\) which lifts to a periodic Schr\"odinger operator 
$$
\mathcal{H}_\Gamma=\mathcal{V}\otimes \mathrm{id}+\sum_{e\in \vec E}w_eA_e\otimes \rho_\Gamma(g_e)\in M_{V}(\mathcal{N}\Gamma).$$
The \textbf{spectral measure} of \(G_\Gamma\) with respect to 
\(\mathcal H_\Gamma\) is defined as
\[
\mu_{G_\Gamma}(B)
=
\frac{1}{|V|}
\dim_{\Gamma} \mathbf 1_B\left( \mathcal{H}_\Gamma\right).
\]
\end{definition}

From \(\Gamma\)-invariance, one can check that
    \[
\mu_{G_\Gamma}(B)
=
\frac{1}{|V|} \sum_{v\in V}
\langle \delta_v\otimes \delta_{g_v}, \mathbf 1_B\left( \mathcal{H}_\Gamma\right) \delta_v\otimes \delta_{g_v}\rangle
\]
holds for arbitrarily chosen \((g_v)\in \Gamma^V\). We also have the following fact, which, in the language of unimodular graphs, means that 
weighted random unimodular graphs of the same law gives the same spectral measure.
\begin{lemma}\label{lm;isodeterminedos}
  Let \(\Gamma_1\) and \(\Gamma_2\) be two groups and \(G\) and \(H\) be two finite graphs. If there is a bijection \(\iota\) from \(V(G)\) and \(V(H)\), such that 
  for each connected component of \(G_{\Gamma_1}\) containing a lift of \(v\in V(G)\),  one can find an isomorphism from \(G_{\Gamma_1}\) to a component of \(H_{\Gamma_2}\) which maps the lift of \(v\) to a  lift of \(\iota(v)\) and preserves the periodic Schr\"odinger operator, then 
  \[\mu_{G_{\Gamma_1}}=\mu_{H_{\Gamma_2}}.\]
\end{lemma}
\begin{proof}
Fix \(r\in \N\). For any \(v\in V(G)\), choose its preimage \(v'\in V(G_{\Gamma_1})\) and preimage \(\iota(v)'\in V(H_{\Gamma_2})\) of \(\iota(v)\in V(H)\) so that the isomorphism in the statement exists. Since \(\langle\delta_{v'}, \mathcal{H}^r_{\Gamma_1} \delta_{v'}\rangle\) and \(\langle\delta_{\iota(v)'}, \mathcal{H}^r_{\Gamma_2} \delta_{\iota(v)'}\rangle\) only depends on the connected components containing \(v'\) and 
\(\iota(v)'\), the isomorphism gives 
\[\langle\delta_{v'}, \mathcal{H}^r_{\Gamma_1} \delta_{v'}\rangle=\langle\delta_{\iota(v)'}, \mathcal{H}^r_{\Gamma_2} \delta_{\iota(v)'}\rangle.\]
By definition of the spectral measure, we have 
\[\int x^r\mathrm{d}\mu_{G_{\Gamma_1}}=\sum_{v\in V_G}\langle\delta_{v'}, \mathcal{H}^r_{\Gamma_1} \delta_{v_1}\rangle \qquad \text{and}\qquad \int x^r\mathrm{d}\mu_{H_{\Gamma_2}}=\sum_{u\in V_G}\langle\delta_{\iota(v)'}, \mathcal{H}^r_{\Gamma_2} \delta_{\iota(v)'}\rangle\]
where we use that \(\iota\) is a bijection. Thus for any \(r\in \N\), 
\[\int x^r\mathrm{d}\mu_{G_{\Gamma_1}}=\int x^r\mathrm{d}\mu_{H_{\Gamma_2}},\]
therefore \(\mu_{G_{\Gamma_1}}=\mu_{H_{\Gamma_2}}\) since bounded measures are characterised by their moments.
\end{proof}

\section{Background on
algebraic matching theory}
\label{sec:comb}
The main combinatorial input that we use to prove Theorem \ref{conj} will be an analysis of matching polynomials and so in this section we present some of their properties. Recall the notation that we use for graphs provided in the introduction.

\begin{definition}
Let $G=(V,E)$ be a finite multi-graph. A \textbf{$k$-matching} in $G$ is a collection of \(k\) independent edges, or equivalently, a $1$-regular subgraph of $G$ with \(k\) edges. The empty subgraph is seen as the unique \(0\)-matching. A matching is called \textbf{maximum} if its edge set has the largest cardinality among all matchings, and is called \textbf{perfect} or a \textbf{\(1\)-factor} if it contains all vertices in \(V\). Denote by $\mathcal{M}_G$ the collection of all matchings of $G$.
\end{definition}

\begin{definition}
    Given a matching \(M\), an \(M\)-\textbf{alternating} walk is a walk that transverses edges that alternate between edges from $M$ and edges not in $M$. Similarly one can define alternating paths and cycles.
\end{definition}

\begin{remark}\label{thm:Berge}
    The word \textbf{maximal} is in the set containment sense in matching theory, that is, removing the matching as a subgraph from $G$ leaves only singletons possibly with self-loops. Therefore, a matching might not be extended to a maximum matching. However, given any matching, there always exists a maximum matching containing all of its vertices, see for example \cite[Corollary 3.1.6]{lovasz2009matching}. 
\end{remark}

The surplus of a set $S$ of vertices in a graph $G$ measures how much its vertex boundary $\partial S$ outnumbers it. That is, $\partial  S$ is the collection of vertices in $G$ not contained in $S$ but are connected to at least one vertex in $S$. We extend this notion to subgraphs of $G$ also. We define the surplus of a graph as follows.

\begin{definition}
    Let \(G=(X\sqcup Y,E)\) be a bipartite graph. The surplus of \(G\) from \(X\) is 
    \[\min_{\varnothing\neq S\subset X} \left( \left|\partial  S\right|-|S| \right).\]
\end{definition}

Bipartite graphs with non-negative surplus will play an important role in this article because they imply the existence of certain matchings. Indeed, we recall the following classic result.

\begin{theorem}[Hall's marriage theorem]
\label{thm:Halls_thm}
  Let \(G=(X\sqcup Y,E)\) be a bipartite graph. There exists a matching
  that contains \(X\) if and only if
  the surplus from \(X\) is non-negative, i.e., 
  \[
    |\partial  S| \geq |S|
    \qquad\text{for every } S\subseteq X.
  \]
  As a result, the surplus from \(X\) is positive if and only if there exists a matching
  that contains \(X\) in \(G\) after deleting any vertex from \(Y\).
\end{theorem}

We now recall the generalised matching polynomial of a graph that relates the matching structure to an ambient Schr\"{o}dinger operator $\mathcal{H}$ on the graph defined as in equation \eqref{eq:Schrodinger_op}. 

\begin{definition}\label{Def:generalisedmatchingpolynomial}
The generalised matching polynomial of $G$ with respect to $\mathcal{H}$, denoted as $m_G$, is defined by
\[
m_G(x)=\sum_{M\in \mathcal{M}_G}(-1)^{|E(M)|}\prod_{e\in \vec{E}(M)}w_e\cdot \prod_{u\in V\setminus V(M)}\left(x-\mathcal{V}_u\right).
\]
If \(G_0\) is a subgraph of \(G\), we abuse notation by denoting the generalised matching polynomial of \(G_0\) with the restriction of \(\mathcal{H}\) as  
\(
m_{G_0}.
\) If the Schr\"{o}dinger operator $\mathcal{H}$ that we are referring to is unambiguous, we will simply write \(m_G\).

\end{definition}
When \(\mathcal{H}=A_G\) is the adjacency matrix, $$m_G(x)=\sum_{k=0}^{\left\lfloor{|V|}/{2}\right\rfloor} (-1)^k\,|\{k\text{-matchings of}\ G\}|\cdot x^{|V|-2k}$$
becomes the classical matching polynomial of the graph $G$. Moreover, \(0\) is a root of the standard matching polynomial if and only if \(G\) has a perfect matching, and this remains true for Schr\"{o}dinger operators whose potential term is zero.

The main tool of analysis of matching polynomials is through a series of identities that they satisfy. We provide a (non-exhaustive) list of these identities that we will make use of through this article. The first and the third identities were established in Heilmann-Lieb \cite[Equation (4.1), Theorem 6.3]{heilmann1972theory}, the second can be found in \cite[Theorem 1.1]{mcsorley2009multivariate}. Recall the conventions for deleting subgraphs and edges as in the introduction. 

\begin{proposition}
        Let \(G\) be a finite multi-graph, then

\begin{enumerate}[(i)]
            \item For any \(v\in V(G)\),
   \begin{equation}\label{eq:vertexresursion}
       m_G(x)=\left(x-\mathcal{V}_v\right)\cdot m_{G\setminus \{v\}}(x)-\sum_{u\sim v, u\neq v} \left(\sum_{e'\in \vec{E};\, o(e)=v,\,t(e)=u}|w_{e}|^2\right)\cdot m_{G\setminus \{v,u\}}(x).
   \end{equation}

   \item 
   For any \(F\subset  E(G)\),
   \begin{equation}\label{eq:multiedgerecursion}
       m_G(x)=\sum_{M \in \mathcal{M}_G,\; E(M) \subset F} (-1)^{|E(M)|}\prod_{e\in \vec{E}(M)}w_e\cdot m_{(G\setminus F)\setminus V(M)}(x).
    \end{equation}

\item
For any \(u\neq v\in V(G)\), 
    \begin{equation}\label{eq:pathidentity}
        m_{G\setminus\{u\}}(x)m_{G\setminus\{v\}}(x)-m_{G}(x)m_{G\setminus\{u,v\}}(x)=\sum_{P\in \mathcal P_{u,v}} \prod_{e\in \vec{E}(P)} w_e \cdot m_{G\setminus P}(x)^2,
    \end{equation}
where \(\mathcal P_{u,v}\) is the collection of paths with endpoints $u$ and $v$; in particular, if $P$ is a path from $u$ to $v$ then both $P$ and its inverse path from $v$ to $u$ reside in $\mathcal{P}_{u,v}$.
\end{enumerate}

Note that \(\prod_{e\in \vec{E}(M)}w_e\) and \(\prod_{e\in \vec{E}(P)} w_e\) are both positive real numbers since both directions of each edge feature in the product.
\end{proposition}

We shall also use the following standard terminology from the generalised Gallai-Edmonds theory for matching polynomials. For a graph \(G\), write \(\operatorname{mult}_\theta(G)\) for the multiplicity of \(\theta\) as a zero of the generalised matching polynomial corresponding to
\(G\). Note that for any \(v\in V(G)\), one always has \[\operatorname{mult}_\theta(G\setminus\{v\})-
\operatorname{mult}_\theta(G)\in \{0,\pm1\},\]
since \(m_{G\setminus\{v\}}\) always interlaces \(m_{G}\) due to Godsil \cite[Corollary 1.3]{godsil1993algebraic}. Thus we have the following classification of vertices in G.

\begin{definition}
 A vertex \(v\in V(G)\) is called \(\theta\)-\textbf{essential (resp., neutral, positive)} if \[
\operatorname{mult}_\theta(G\setminus\{v\})-
\operatorname{mult}_\theta(G)=-1\,(\text{resp.,}\,0,1).
\]
Let \(D_\theta=D_\theta(G)\) denote the set of \(\theta\)-essential vertices of \(G\), and with abuse of notation, the induced subgraph on this vertex set too. A vertex \(u\in \partial  D_\theta(G)\) is called \(\theta\)-\textbf{special}. A connected graph is called \(\theta\)-\textbf{critical} if all of its vertices are \(\theta\)-essential.
\end{definition}

\begin{remark}
    In the case of \(\theta=0\) for the adjacency operator, a \(0\)-critical graph is traditionally called factor-critical. 
    The notion of \(\theta\)-critical graph is the counterpart of \(\theta\)-irreducible graph corresponding to the characteristic polynomial. When the graph is a tree, these two are the same.
\end{remark}

\begin{lemma}[{\cite[Lemma 3.1]{godsil1995algebraic}}]\label{lm:essenexistence}
    If \(m_G(\theta)=0\), then there exists a \(\theta\)-essential vertex in \(G\).
\end{lemma}

There is a useful result due to Godsil on \(\theta\)-critical graphs, which is a direct consequence of \eqref{eq:pathidentity}.

\begin{lemma}[{\cite[Lemma 3.8]{godsil1995algebraic},\cite[Lemma 20]{spier2025eigenvalues}}] \label{lemma:pathincritical}
    Let \(K\) be a \(\theta\)-critical graph.
    For any \(u\neq v\in V(K)\), there exists a path \(P\) connecting $u$ and $v$ such that \(m_{K\setminus P}(\theta)\neq 0\).
\end{lemma}

Now we introduce the generalised Gallai-Edmonds structure theorem. The original theorem which only concerned the structure of maximum matchings can be found in \cite[Section 3.2]{lovasz2009matching}, for example. The extension based on the structure of critical components for the standard matching polynomial is due to Godsil \cite[Theorem 4.2]{godsil1995algebraic} and 
Ku and Chen \cite[Theorem 1.7, Lemma 2.4]{ku2010analogue}. In the weighted setting used here, the result is due to Ku and Wong \cite[Lemma 4.1 and Theorem 4.13]{ku2013gallai}; see also the refinement of Spier \cite{spier2023refined}. The first part of the structure theorem is usually referred to as \textbf{Gallai's lemma}, while the second is referred to as the \textbf{stability lemma}.

\begin{theorem}[The generalised Gallai-Edmonds theorem]\label{thm:generalisedGallaiEdmonds}
Let \(G\) be a finite graph,

\begin{enumerate}[(i)]
    \item If \(G\) is connected and \(\theta\)-critical, then
    \(
    \operatorname{mult}_\theta(G)=1.
    \) 
    
    \item If \(v\) is \(\theta\)-special, then for any \(u\neq v\), \(u\) is \(\theta\)-essential (resp. neutral, positive) in \(G\) if and only if \(u\) is \(\theta\)-essential (resp. neutral, positive) in \(G\setminus\{v\}\).
\end{enumerate}
\end{theorem}

We also use the following property about $\theta$-special vertices due to Godsil \cite[Corollary 4.3]{godsil1995algebraic} and Ku and Wong \cite[Lemma 4.1]{ku2013gallai}.

\begin{proposition}
    If \(v\) is \(\theta\)-special, then it is \(\theta\)-positive. 
\end{proposition}

Theorem \ref{thm:generalisedGallaiEdmonds} immediately implies the following result about the multiplicity of $\theta$ as a matching polynomial root.

\begin{proposition}
\label{prop:mult-roots}
    Let \(G\) be a finite graph. For any \(\theta\in \mathbb R\), each component of \( D_{\theta}\) is \(\theta\)-critical and \[\mathrm{mult}_{\theta}(G)=\operatorname{cc} D_{\theta}-\left|\partial   D_{\theta}\right|.\]

\end{proposition}

We note the following result about $\theta$-critical paths which slightly generalises \cite[Theorem 45]{spier2025eigenvalues} with almost the same proof. This will be used later in Section \ref{sec:monotone}.

\begin{lemma}
\label{lem: path lemma}
    If there is a path $P$ in \(G\) such that \(m_{G\setminus P}(\theta)=0,\)
    then either it contains a vertex \(v\in V(P)\) such that \(m_{G\setminus\{v\}}(\theta)=0,\) or there exists a cycle, $\gamma$ in $G$, such that 
    \[m_{G\setminus\gamma}(\theta)\neq 0.\]
\end{lemma}

\begin{proof}
Suppose that $\tilde{P}=\{v_1,\ldots,v_k\}$ is a shortest length sub-path of $P$ for which
$m_{G\setminus\tilde{P}}(\theta)=0$.
If $k=1$, then we are done so assume that $k\geq 2$. Since $\tilde{P}$ is the shortest sub-path, \(m_{G\setminus\{v_2,\cdots,v_k\}}(\theta)\neq 0\). Therefore, applying \eqref{eq:vertexresursion} to the graph \(G \setminus\{v_2,\cdots,v_k\}\) at vertex \(v_1\), there exists a neighbour \(a\) of \(v_1\), \(a\neq v_1,\cdots, v_k\) such that \[m_{G\setminus\{a, v_1,\cdots,v_k\}}(\theta)\neq 0. \] Similarly, there exists a neighbour \(b\) of \(v_k\), \(b\neq v_1,\cdots, v_k\) such that \[m_{G\setminus\{b, v_1,\cdots,v_k\}}(\theta)\neq 0.\] 
Now, applying \eqref{eq:pathidentity} to the graph \(G_{\mathrm{int}}:=G\setminus\{v_1,\cdots,v_k\}\), we have 
\[0\neq m_{G_\mathrm{int}\setminus\{a\}}(\theta)m_{G_\mathrm{int}\setminus\{b\}}(\theta)=\sum_{P' \in \mathcal P_{a,b}} \prod_{e\in \vec E} w_e \cdot m_{G_\mathrm{int}\setminus P'}(\theta).\]
As a result, there exists a path \(P'\subseteq G_\mathrm{int}\) such that \(m_{G_\mathrm{int}\setminus P'}(\theta)\neq 0\). Now, let \(\gamma\) be the cycle in \(G\) formed by \(P_0\), \(\tilde{P}\) and two edges connecting \(a,v_1\) and \(b,v_2\). Then 
\[G_\mathrm{int}\setminus P'=G\setminus \gamma\] 
and so \(m_{G\setminus\gamma}(\theta)\neq 0\). 
\end{proof}

An immediate corollary is the following. 
\begin{corollary}[{\cite[Theorem 45]{spier2025eigenvalues}}] \label{cor:cycleincritical}
    Let \(K\) be a \(\theta\)-critical graph that is not a tree, then there exists a cycle \(\gamma\) in $K$ such that \(m_{K\setminus\gamma}(\theta)\neq 0\).
\end{corollary}

\section{Spectral measure of atoms for maximal abelian covers}

\label{sec:abelian}

In this section, we analyse the atoms  of the  spectral measure of $\mathcal{H}^\ab$, the lift of $\mathcal{H}$ to the maximal abelian cover of $G = (V,E)$. Our main result is the following exact expression which is a refinement of the criterion of \(\ell^2\)-eigenvalues for the maximal abelian cover due to Magee, Sabri and the second and third authors.

Throughout, \(\mathcal C(G)\) will denote the set of
\(2\)-regular subgraphs of \(G\), including the empty subgraph.

\begin{theorem}\label{prop:multicriterion}
Let \(G = (V,E)\) be a finite graph. Then, for every
\(\theta\in\mathbb R\),
\[
\mu_{G^\mathrm{ab}}(\{\theta\})
=
\frac{1}{|V|}
\min_{\gamma\in\mathcal C(G)}
\mathrm{mult}_{\theta}
(G\setminus\gamma).
\]
\end{theorem}

This motivates the following definition.

\begin{definition}
A finite graph $G$ satisfies \(\mathrm{LMST}(\theta)\) if \[m_{G\setminus\gamma}(\theta)=0 \qquad \forall \gamma \in \mathcal C(G). \]
\end{definition}
 The naming of this property is motivated by the that fact that a graph satisfies \(\mathrm{LMST}(\theta)\) if and only if \(\theta\) is an eigenvalue of $\mathcal{H}^\ab$, as proved in \cite{LiMageeSabriThomas}. 

The proof of Theorem \ref{prop:multicriterion} will borrow some ideas from \cite{LiMageeSabriThomas} (see also Example \ref{def of Gper} above). To this end, we introduce the following notation.

For a finite graph $G=(V,E)$, let $E_+$ be an assigment of orientation to each edge in $E$. Denote by \(\mathbb{Z}^{E_+}=\{(n_e)_{e\in E_+}\}\) the free abelian group generated by \(E_+\). Furthermore, denote by \(\mathbb T_G=\{(z_e)_{e\in E_+}\}\) the dual torus of \(\mathbb{Z}^{E_+}\) under the Fourier transform. For any induced subgraph $H$ of $G$, we will always see \(\mathbb T_H\) as a direct summand of \(\mathbb T_G\).

We apply the general theory introduced in Section \ref{sec:op_algs} to this voltage assignment, again, recall the discussion in Example \ref{def of Gper}.

\begin{definition}[The magnetic Schr\"odinger operator]\label{def:phipoly}
Given a Schr\"{o}dinger operator $\mathcal{H}$ on $\ell^2(V)$ of the form 
    $$\mathcal{H}=\sum_{e\in E_+}w_eA_e+\overline{w_e}A_{\overline{e}}+\mathcal{V},$$
we define 
    $$\mathcal{H}(z)=\sum_{e\in E_+}w_ez_eA_e+\overline{w_e}\,\overline{z_e}A_{\bar{e}}+\mathcal{V} \qquad \forall z=(z_e)_{e\in E_+}\in \T_G,$$
and set \[\phi_G=\phi_G(x,z):=\det(xI-\mathcal H(z)).\]

\end{definition}

The equivalence of $\theta$ being an $\ell^2$-eigenvalue in the spectrum of $\mathcal{H}$ on $G^\ab$ and $G$ satisfying LMST($\theta$) proven in \cite{LiMageeSabriThomas} relies on the observation eluded to in Example \ref{def of Gper} to see that $\theta$ is an $\ell^2$-eigenvalue if and only if 
\[\phi_G(\theta,z)\equiv 0 \qquad \forall z\in \T_G.\]

Recall the derived voltage graph \(G^\per\) built from the edge assignment from $\mathbb{Z}^{E_+}$ as defined in Example \ref{def of Gper}. We denote by $\mathcal{H}^\per$ the Schr\"{o}dinger operator $\mathcal{H}$ lifted from $G$ to $G^\per$; it is $\mathcal{H}_{\mathbb{Z}^{E_+}}$ in the language of Section \ref{sec:op_algs}. We denote its spectral measure by $\mu_{G^\per}$ which satisfies
    $$\mu_{G^\per} = \mu_{G^\ab},$$
since $G^\per$ is just countably many copies of $G^\ab$ as proven in \cite{LiMageeSabriThomas} and we can use Lemma \ref{lm;isodeterminedos}.

From Section \ref{sec:op_algs}, $\mathcal{H}^\per$ acts on 
    $$\ell^2(V(G^\per))\cong \ell^2(V)\otimes \ell^2(\Z^{E_+})$$
by
    $$\mathcal{V}\otimes I+\sum_{e\in E_+} w_eA_e\otimes \rho_{\mathbb{Z}^{E_+}}(E_e)+ w_{\bar{e}}A_{\bar{e}}\otimes \rho_{\mathbb{Z}^{E_+}}(E_{\bar{e}}),$$
where $E_e$ is the vector with $-1$ in the coordinate labelled by $e$ and zero elsewhere, and $E_{\overline{e}}$ is the same but with $-1$ replaced by $1$. 

The following can be obtained as a consequence of Floquet theory (see, for example \cite[Lemma 2.1]{sabri2023flat}) and is essentially contained in \cite[Proposition 2.4]{LiMageeSabriThomas}. We provide an abstract harmonic analysis proof here.

\begin{lemma}
    We have an isomorphism
        $$\ell^2(V)\otimes L^2(\Z^{E_+})\cong \int_{\T_G}^\oplus \ell^2(V)\,\mathrm{d}z,$$
    where $\mathrm{d}z$ is the uniform measure. Moreover, under this isomorphism, $\mathcal{H}^\per$ intertwines to 
        $$\int_{\T_G}^\oplus \mathcal{H}(z) \mathrm{d}z.$$
\end{lemma}

\begin{proof}
    For $v\in V$ and $(n_e)_{e\in E_+}$, let $\delta_v$ and $\delta_{(n_e)_e}$ be the delta functions in $\ell^2(V)$ and $\ell^2(\mathbb{Z}^{E_+})$ respectively. The pure tensors $\delta_v\otimes\delta_{(n_e)_{e}}$ span $\ell^2(V)\otimes \ell^2(\Z^{E_+})$. Now consider the map to $\int_{\T_G}^\oplus \ell^2(V)\,\mathrm{d}z$ given by
        $$\delta_v\otimes\delta_{(n_e)_{e}} \mapsto (e^{i (n_e)_e \cdot (\theta_e)_e}\delta_v)_{(e^{i\theta_e})_{e\in E_+}}.$$
    By extending linearly, this gives an isomorphism. 

    On these pure tensors, $\mathcal{H}^\per$ acts by
    $$\delta_v\otimes\delta_{(n_e)_{e}} \mapsto \mathcal{V}_v\delta_v\otimes\delta_{(n_e)_e}+\sum_{\substack{f\in E_+ \\ t(f)=v}} w_f\delta_{o(f)}\otimes \delta_{(n_{e})-E_{f}}+\sum_{\substack{f\in E_+ \\ o(f)=v}}w_{\bar{f}}\delta_{t(f)}\otimes\delta_{(n_{e})-E_{\bar{f}}}.$$
The latter is mapped to 
$$(e^{i (n_e)_e \cdot (\theta_e)_e}\mathcal{V}_v\delta_v+e^{i (n_e)_e \cdot (\theta_e)_e}\sum_{\substack{f\in E_+ \\ t(f)=v}} e^{i\theta_f}w_f\delta_{o(f)}+e^{i (n_e)_e \cdot (\theta_e)_e}\sum_{\substack{f\in E_+ \\ o(f)=v}} e^{-i\theta_f}w_{\bar{f}}\delta_{t(f)})_{(e^{i\theta_e})_{e\in E_+}},$$
which agrees with the action of $\int_{\T_G}^\oplus \mathcal{H}(z) \mathrm{d}z$.\end{proof}

We now prove Theorem \ref{prop:multicriterion}.

\begin{proof}[Proof of Theorem \ref{prop:multicriterion}]

By functional calculus, for any $\theta\in\R$, the operator
\(\mathbf 1_{\{\theta\}}(\mathcal H(z))\) is the orthogonal projection onto
\(\ker(\mathcal H(z)-\theta I)\).
Thus we have
\[
\operatorname{Tr}\bigl(\mathbf 1_{\{\theta\}}(\mathcal H(z))\bigr)
=
\dim\ker(\mathcal H(z)-\theta I).
\]
Then, by the equivalence of $\mu_{G^\ab}$ and $\mu_{G^\per}$, 

$$
\begin{aligned}
    \mu_{G^\mathrm{ab}}(\{\theta\})
=
\mu_{G^\mathrm{per}}(\{\theta\})
&=\frac{1}{|V|}
\sum_{v\in V}
\left\langle
\mathbf 1_{\{\theta\}}(\mathcal H^\per)\delta_{v} \otimes \delta_0,
\delta_{v} \otimes \delta_0
\right\rangle\\
&=\frac{1}{|V|}
\sum_{v\in V}
\left\langle
\mathbf 1_{\{\theta\}}\left(\int_{\T_G}^\oplus \mathcal{H}(z) \mathrm{d}z\right) \mathrm{1}\cdot \delta_{v} ,
\mathrm{1}\cdot \delta_{v} 
\right\rangle\\
&=\frac{1}{|V|}
\sum_{v\in V}
\left\langle
\left(\int_{\T_G}^\oplus \mathbf 1_{\{\theta\}}\left(\mathcal{H}(z) \right)\mathrm{d}z\right) \mathrm{1}\cdot \delta_{v} ,
\mathrm{1}\cdot \delta_{v} 
\right\rangle\\
&=\frac{1}{|V|}
\sum_{v\in V}
\int_{\T_G}\langle \mathbf 1_{\{\theta\}}\mathcal{H}(z) \delta_v,\delta_v\rangle\, \mathrm{d}z\\
&=\frac{1}{|V|} \int_{\T_G}\operatorname{Tr}\bigl(\mathbf 1_{\{\theta\}}(\mathcal H(z))\bigr)\,\mathrm{d}z\\
&=
\frac{1}{|V|}
\int_{\T_G}
\dim\ker(\mathcal H(z)-\theta I)\,\mathrm{d}z.
\end{aligned}
$$

By the determinant expansion of
\cite[Proposition 3.8]{LiMageeSabriThomas}, we have
\begin{equation}\label{eq:detmatch}
\phi_G(x,z)
=
\sum_{\vec\gamma }
(-1)^{\operatorname{cc}(\vec\gamma)}
m^{\mathcal H}_{G\setminus\vec\gamma}(x)\,
w_{\vec\gamma}(z),
\end{equation}
where the sum is over oriented \(2\)-regular subgraphs of \(G\), and the
functions \(w_{\vec\gamma}\) given by the product of $w_ez_e$ over each oriented edge $e$ in $\vec{\gamma}$ are mutually orthogonal on the torus $\mathbb{T}_G$. Note that
\(G\setminus\vec\gamma\) and
\(m^{\mathcal H}_{G\setminus \vec\gamma}\) depend only on the underlying
unoriented \(2\)-regular subgraph.

Let
\(
k:=
\min_{\gamma\in\mathcal C(G)}
\mathrm{mult}_{\theta}
(G\setminus\gamma).
\)
Thus every \(m^{\mathcal H}_{G\setminus\gamma}(x)\) is divisible by
\((x-\theta)^k\) and so \(\phi_G(x,z)\) is also for every
\(z\), and therefore
\[
\dim\ker(\mathcal H(z)-\theta I)\ge k,
\;\;
\forall z\in\T_G \qquad \text{and} \qquad
\mu_{G^\mathrm{ab}}(\{\theta\})\ge \frac{k}{|V|}.
\]
Now, choose \(\gamma_0\in\mathcal C(G)\) such that
\(
\mathrm{mult}_{\theta}
(G\setminus\gamma_0)=k,
\)
and choose an orientation \(\vec\gamma_0\) of \(\gamma_0\). Define
\[
Q(z):=
\left.
\frac{\phi_G(x,z)}{(x-\theta)^k}
\right|_{x=\theta}.
\]
Using the expansion above,
\[
Q(z)
=
\sum_{\vec\gamma}
(-1)^{\operatorname{cc}(\vec\gamma)}
\left.
\frac{m^{\mathcal H}_{G-\vec\gamma}(x)}
{(x-\theta)^k}
\right|_{x=\theta}
w_{\vec\gamma}(z).
\]
The coefficient of \(w_{\vec\gamma_0}(z)\) is non-zero. Since the functions
\(w_{\vec\gamma}\) are mutually orthogonal, \(Q\) is not identically zero.
Hence its zero set in \(\T_G\) has Haar measure zero. Note that 
\(\dim\ker(\mathcal H(z)-\theta I)> k\) if and only if \(Q(z)=0,\)
thus we have
\(
\dim\ker(\mathcal H(z)-\theta I)=k
\)
for almost every \(z\in\T_G\). Consequently,
\[
\mu_{G^\mathrm{ab}}(\{\theta\})
=
\frac{1}{|V|}
\int_{\T_G}
\dim\ker(\mathcal H(z)-\theta I)\,\mathrm{d}z
=
\frac{k}{|V|}.
\]
The final equivalence follows immediately from the formula.
\end{proof}

We conclude this section with a lemma regarding the decomposition of $\phi_G$ to subgraphs connected by a bridge in $G$ that will be of importance later. Recall that a bridge is an edge in $G$ whose removal disconnects the graph. {This is a direct consequence of determinant expansion and we omit the proof.}

\begin{lemma}\label{lemma:turnonoff}
    Let \(G\) be a finite graph with a bridge \(e\), and let \(K\) and \(L\) be the components of \(G\setminus e\) with \(u\in V(K) \) and \(v\in V(L)\) the two ends of the bridge. Then 
    \[\phi_K\phi_L-|w_e|^2\phi_{K\setminus\{u\}}\phi_{L\setminus\{v\}}=\phi_G,\]
    where the magnetic Schr\"odinger operator on \(K\) and \(L\) is the restriction of the magnetic Schr\"odinger operator from \(G\).
\end{lemma}

\section{The lower bound: a generalised Spier's theorem}

\label{sec:spier}

In this section we aim to prove the lower bound in Theorem \ref{conj} \emph{without} using the criteria of Salez \cite[Theorem 2]{salez2020spectral} and Banks, Garza-Vargas and Mukherjee \cite[Corollary 3.4]{banks2022point} on \(\ell^2\)-eigenvalues of universal covering trees. In fact we reprove their criteria as a by-product of our arguments. The lower bound also serves as a multiplicity-refined version of Spier's theorem when $G_\Gamma$ corresponds to $G^\mathrm{uni}$. 

\begin{theorem}\label{thm:lowerbound} 
    Let \(G_\Gamma\) be a connected normal cover of a finite graph \(G\) with deck group \(\Gamma\). For any \(\theta\in \mathbb R\), 
\[\mu_{G^\mathrm{ab}}(\{\theta\}) \leq \mu_{G_\Gamma}(\{\theta\}).\]
\end{theorem}

Towards this, we introduce the following associated graph from $G$ and $\theta\in\R$.

\begin{definition}[The associated graph]\label{def:associatedbipartite}
    Given a finite graph $G$ and $\theta\in\R$, the associated graph \({G}_{\theta}\) is a bipartite graph associated with \(G\) with vertex set \[V(G_{\theta})=\partial  D_\theta \sqcup \{K_i: K_i \text{ is a component of}\;D_\theta\}.\]
We abuse the notation and define the edge set \(E(G_{\theta})\) as the edge boundary of \(D_\theta\) in \(G\), where each edge between \(u\in \partial  D_\theta\) and a connected component \(K_j\) of \(D_\theta\) corresponds to an edge connecting them in the original graph \(G\). 
\end{definition}

The following result is due to Spier \cite[Corollary 32]{spier2023refined} but we provide a different proof that highlights an application of \eqref{eq:multiedgerecursion} that will be used repeatedly throughout the article as mentioned in the introduction.

\begin{lemma}\label{lemma:Halltype}
    Let \(G\) be a finite graph such that \(\mathrm{mult}_{\theta}(G)>0\) for some \(\theta\in \mathbb R\). Then 
    \(G_\theta\) has positive surplus from 
    \(\partial  D_\theta\).
\end{lemma}

\begin{proof}
    By Theorem \ref{thm:Halls_thm} it is equivalent to show that after deleting any component in $D_\theta$, there exists a matching in $G_\theta$ that contains all of $\partial D_\theta$. Suppose for a contradiction that this is not the case, and let $K_0$ be component in $D_\theta$ for which this fails. 

    It means that any matching in $G_\theta$ that does not contain the component $K_0$ contains at most $|\partial  D_\theta|-1$ vertices in $\partial  D_\theta$ and hence at most $|\partial  D_\theta|-1$ components in $D_\theta$ outside of $K_0$. 
    
    Fix any vertex \(v\) in the component \(K_0\). Let \(F\) be the edges in $G$ between \(\partial  D_\theta\) and vertices in components of \(D_\theta\) outside of \(K_0\). Recall that by \eqref{eq:multiedgerecursion},
    \[
       m_{G\setminus\{v\}}(x)=\sum_{M \in \mathcal{M}_{G\setminus\{v\}},\; E(M) \subset F} (-1)^{|E(M)|}\prod_{e\in \vec{E}(M)}w_e\cdot m_{(G\setminus F)\setminus (V(M)\sqcup\{v\})}(x).
    \]

    Note that any matching considered in this formula can contain vertices from at most $|\partial  D_\theta|-1$ different components in $D_\theta$ with $K_0$ removed. Indeed if not, then we could pick a single edge in this matching that land in these distinct components and form a matching in $G_\theta$ that contains at $|\partial  D_\theta|$ components in $D_\theta$ without $K_0$, contradicting the selection of $K_0$. 

    Now given such a matching $M$, the components of $D_\theta$ that are not touched by $M$ become isolated components in $(G\setminus F)\setminus(V(M)\sqcup\{v\})$ because $F$ contains all edges that connect them to the rest of the graph. By Proposition \ref{prop:mult-roots}, these components are $\theta$-critical and so have $\theta$ as a root of their matching polynomial.
    
    The matching polynomial of a disconnected graph is the product of the matching polynomials of its components and so for each unmatched component, $m_{(G\setminus F)\setminus (V(M)\sqcup\{v\})}(x)$ gains a multiplicity of at least $1$ for $\theta$ as a root. By the above, the number of unmatched components is at least
        $$|\operatorname{cc}D_\theta |-1-(|\partial  D_\theta|-1)=\operatorname{cc} D_\theta-|\partial  D_\theta |,$$
    thus, 
        $$\mathrm{mult}_{(G\setminus F)\setminus(V(M)\sqcup\{v\})}(\theta)\geq \operatorname{cc} D_\theta-|\partial  D_\theta |,$$
    and by the above formula, the same is true for $\mathrm{mult}_{G\setminus\{v\}}(\theta)$. On the other hand, $v$ is $\theta$-essential in $G$ and so
        $$\mathrm{mult}_{G}(\theta)-1=\mathrm{mult}_{G\setminus\{v\}}(\theta)\geq \operatorname{cc} D_\theta-|\partial  D_\theta |,$$
    but this gives a contradiction to Proposition \ref{prop:mult-roots}.
\end{proof}

We will see that the important components of $D_\theta$ that will be related to the spectral measure of atoms are those that are acyclic. This motivates the following definition.

\begin{definition}\label{def:refinedAomoto}
    Let \(\mathcal{T}_{\theta}\) be the collection of induced subgraphs of \(G\), of which each connected component of is an acyclic \(\theta\)-critical component of \(G\). 
Define 
\[\kappa_{\theta}:=\max_{S\in \mathcal{T}_\theta}\left(\operatorname{cc} S-|\partial  S|\right).\]
\end{definition}

We now prove the following.

\begin{lemma}
\label{lem: kappa_lb}
Let \(G = (V,E)\) be a finite graph such that \(\mathrm{mult}_{\theta}(G)>0\) for some \(\theta\in \mathbb R\). We have
    \[\kappa_{\theta}\geq \min_{\gamma \in \mathcal C(G)}
\operatorname{mult}_\theta(G\setminus\gamma)
.\]
In particular, for every
\(\theta\in\mathbb R\),
\[
\mu_{G^\mathrm{ab}}(\{\theta\})
\leq \frac{\kappa_\theta}{|V|}.
\]
\end{lemma}

\begin{proof}
    Suppose that \(S_0\in \mathcal T_{\theta}\) achieves the maximum in the definition of $\kappa_\theta$. We start with two claims on the matching structure of the bipartite graph \(G_\theta\) constructed in Definition \ref{def:associatedbipartite}.

    \textbf{Claim 1.} {\em Either \(\partial  S_0\) is empty or it satisfies the following Hall marriage theorem type condition: 
    Any \(I \subset \partial  S_0\) connects to at least \(|I|\) components of \(S_0\).}
    
    \textit{Proof of Claim 1.} Denote the number of components of \(S_0\) connecting to \(I\) by \(k\), and let \(S_0'\) be the union of all components of \(S_0\) that are not connected to \(I\). By definition, we have 
    \[\operatorname{cc} S_0'= \operatorname{cc} S_0-k\qquad \text{and} \qquad \partial  S_0'=\partial  S_0\setminus I,\]
    where the boundaries are taken in $G$. Therefore, we have 
    \[\operatorname{cc} S_0-\left|\partial  S_0\right|\geq \operatorname{cc} S'_0-\left|\partial  S'_0\right|=\operatorname{cc} S_0-\left|\partial  S_0\right|+|I|-k,\]
    where the first inequality follows since \(S_0\) is a maximizer. Therefore \(k\geq |I|\) as desired.

    \textbf{Claim 2.} {\em Either the set of acyclic components of \(D_\theta(G)\setminus S_0\) is empty or it satisfies a Hall marriage theorem type condition: Any \(m\) acyclic components of \(D_\theta(G)\setminus S_0\) connect to at least \(m\) vertices in \(\partial  D_\theta(G)\setminus \partial  S_0\) in the graph \(G\).}
    
    \textit{Proof of Claim 2.} Suppose the claim is false, then there exists \(m\) acyclic components of \(D_\theta\setminus S_0\) that connect to at most \(m-1\) vertices in \(\partial  D_\theta\setminus \partial  S_0\). Let \(S_0'\) be the union of these components along with \(S_0\), then \[\operatorname{cc} S_0'= \operatorname{cc} S_0+m\qquad \text{and} \qquad \left|\partial  S_0'\right|\leq\left|\partial  S_0\right|+m-1.\]
    Therefore, we have 
    \(\operatorname{cc} S'_0-\left|\partial  S_0'\right|\geq \operatorname{cc} S_0-\left|\partial  S_0 \right|+1\), but this contradicts the fact that \(S_0\) is a maximizer.

    We now apply Theorem \ref{thm:Halls_thm}. By Claim 1, we can find a matching in $G_\theta$ (and hence in $G$) that contains all vertices in $\partial S_0$. Similarly, by Claim 2, we can find a disjoint matching in $G_\theta$ that contains all acyclic components in $D_\theta(G)\setminus S_0$.

    Due to the disjointness, we can combine these matchings into a single matching and then use Remark \ref{thm:Berge} to obtain a maximum matching $M_\mathrm{max}$ in $G_\theta$ containing at least $|\partial  S_0|$ components of $S_0$ and all acyclic components of $D_\theta\setminus S_0$. Since by Lemma \ref{lemma:Halltype}, $G_\theta$ has positive surplus from $\partial  D_\theta$, it contains a matching containing all of $\partial  D_\theta$ by Theorem \ref{thm:Halls_thm}. In particular, a maximum matching such as $M_\mathrm{max}$ contains all vertices in $\partial  D_\theta$. 
    
    Let \(Q_1, \cdots, Q_t\) be the \(\theta\)-critical components in \(G_\theta\) that are not contained in $M_\mathrm{max}$ that contain a cycle; in particular, they are contained in $D_\theta\setminus S_0$. By Corollary \ref{cor:cycleincritical}, in each \(Q_i\) we can find a cycle \(\gamma_i\) such that \(\mathrm{mult}_{Q_i\setminus \gamma_i}(\theta)=0.\) Define the degree-\(2\) subgraph \(\gamma\) in \(G\) by
    \(\gamma=\cup_{i=1}^t\gamma_i.\)
    We now show that 
    \[\mathrm{mult}_{G\setminus\gamma}(\theta)= \kappa_{\theta}.\]

    Note that the graph \(G\setminus\gamma\) still contains all vertices from \(\partial  D_{\theta} \) and hence all vertices from \(\partial  S_0\). It also contains the maximizer \(S_0\) as a subgraph.
    Let \(F\) be the edges between vertices \(\partial  D_{\theta}\) and \( D_{\theta}\setminus\gamma\) in \(G\setminus\gamma\) (here we are thinking of $D_\theta$ as the induced subgraph). By \eqref{eq:multiedgerecursion}, we have  
    \begin{equation}
    \label{eq:poly-of-Ggamma2}
       m_{G\setminus\gamma}(x)=\sum_{M \in \mathcal{M}_{G\setminus\gamma},\; E(M) \subset F} (-1)^{|E(M)|}\prod_{e\in \vec{E}(M)}w_e\cdot m_{G\setminus (F\cup \gamma \cup V(M))}(x).
    \end{equation}
    We now have to consider several cases for the matching $M\in\mathcal{M}_{G\setminus\gamma}$.
    \begin{itemize}
        \item \textbf{Case 1}. \(M\) contains vertices from at most \(\left|\partial  S_0\right|-1\) components of \(S_0\).
        Since components of \(S_0\) connect to the rest of \(G\setminus\gamma\) only through edges in $F$, there are at least $$\operatorname{cc}S_0-|\partial  S_0|+1\geq \kappa_\theta +1$$ components of $S_0$ that are isolated components in \(G\setminus(\gamma\cup F \cup V(M))\). Since each of these components is \(\theta\)-critical so have $\theta$ as a root of their matching polynomial and the matching polynomial of a disconnected graph is the product of the matching polynomials of the connected components, we have 
        \[\mathrm{mult}_{G\setminus(\gamma\cup F \cup V(M))}(\theta)\geq \kappa_\theta+1.\]

        \item \textbf{Case 2}. \(M\) contains vertices from \(\left|\partial  S_0\right|\) components of \(S_0\) and hence all vertices in \(\partial  S_0\). In this case, \(G\setminus(\gamma\cup F\cup V(M))\) always contains \(\kappa_\theta\) isolated components that are components of \(S_0\). We now split into two possibilities:
            \begin{itemize}
                \item \textbf{Case 2a}. There exists a component $K$ in $D_\theta$ that is not a component of $S_0$ and is not one of the $Q_i$ such that $M$ does not contain one of its vertices. Then in $G\setminus(\gamma\cup F \cup V(M))$, $K$ is an isolated connected component that is $\theta$-critical. Together with the $\kappa_\theta$ isolated $\theta$-critical components of $S_0$ above, and using the fact that matching polynomials of disconnected graphs are the product of the matching polynomials of the components, we obtain
                \[\mathrm{mult}_{G\setminus (\gamma\cup F\cup V(M))}(\theta)\geq \kappa_\theta+1.\]

                \item \textbf{Case 2b}. $M$ contains a vertex from every component $K$ (if any exist) that are not components of $S_0$ nor are one of the $Q_i$. Note that $M_\mathrm{max}$ is an example of such a matching. 
                In this case, because $M$ contains vertices from $|\partial  S_0|$ components of $S_0$ as well as vertices from all components of $D_\theta$ other than $S_0$ and the $Q_i$, it contains at least as many vertices in $\partial  D_\theta$ as $M_\mathrm{max}$ does, which means it contains all of them (because $M_\mathrm{max}$ does), and they are connected to distinct components in $D_\theta$.
                The connected components of $G\setminus(\gamma\cup F\cup V(M))$ are of the following type:
                    \begin{enumerate}[(i)]
                        \item Connected components of $G\setminus(\partial  D_\theta \cup D_\theta)$, none of which will have $\theta$ as a matching polynomial root by construction.
                        \item Connected components of $D_\theta$ that are not $Q_i$ or in $S_0$ and have a vertex missing (the vertex contained in $M$). These were $\theta$-critical graphs before the removal of the vertex from $M$ and hence $\theta$ had multiplicity one as a matching polynomial root by Theorem \ref{thm:generalisedGallaiEdmonds} which is now reduced to zero as the removed vertex is $\theta$-essential.
                        \item Connected components of $Q_i\setminus\gamma_i$ which have do not have $\theta$ as a matching polynomial root.
                        \item The $\kappa_\theta$ components of $S_0$ that have no vertex contained in $M$ and so remain as $\theta$-critical and so contribute a total of $\kappa_\theta$ to the multiplicity of $\theta$ as a root of the matching polynomial of $G\setminus(\gamma\cup F\cup V(M))$. 
                    \end{enumerate}
                Altogether, we see that
                    \[\mathrm{mult}_{G\setminus (\gamma\cup F\cup V(M))}(\theta)= \kappa_\theta.\]
                Moreover, $|E(M)|$ is the same for all of these types of matchings, it is equal to $|\partial  D_\theta|$.
            \end{itemize}
            So, every term on the right hand side of \eqref{eq:poly-of-Ggamma2} has $\theta$ as a root with multiplicity at least $\kappa_\theta$. The matchings corresponding to the final possibility above give polynomials $m_{G\setminus (F\cup \gamma \cup V(M))}$ whose $\theta$ root multiplicity is equal to $\kappa_\theta$; moreover they all are monic of the same degree and they occur in the expansion of \eqref{eq:poly-of-Ggamma2} with the same sign (since $|E(M)|$ is constant) and there is at least one of them (due to $M_\mathrm{max}$). Note also that the factor $\prod_{e \in \vec E (M)} w_e$ is always a positive real number since whenever \(w_e\) appears in the product, so does \(w_{\bar{e}}=\overline{w_e}\). Thus, when we divide \eqref{eq:poly-of-Ggamma2} through by $(x-\theta)^{\kappa_\theta}$, the limit $x\to\theta$ exists and is non-zero and so the conclusion follows. Here we use the fact that the factor $\prod_{e \in \vec E (M)} w_e$ is always a positive real number since whenever \(w_e\) appears in the product, so does \(w_{\bar{e}}=\overline{w_e}\).
    \end{itemize}
It concludes the proof of the lemma. \end{proof}

Towards proving Theorem \ref{thm:lowerbound}, we introduce the notion of Aomoto sets as in \cite{banks2022point}. Graphs $S$ in \(\mathcal T_\theta\) defined in Definition \ref{def:refinedAomoto} are Aomoto sets in the case when their deficit $\operatorname{cc}S-|\partial  S|>0$.

\begin{definition}\label{def:Aomoto}
For \(\theta\in\mathbb R\), an induced subgraph \(S\) is called a $\theta$-Aomoto set \(\mathcal A_\theta(G)\) if the following conditions hold:
\begin{enumerate}[(i)]
    \item \(S\) is a forest;
    \item \(\theta\) is a zero of the
    matching polynomial of each connected component of \(S\);
    \item
    \(
    \operatorname{cc}S>|\partial  S|.
    \)
\end{enumerate}
The collection of all $\theta$-Aomoto sets is denoted $\mathcal{A}_\theta (G)$.
\end{definition}
By definition, we have 
\[\kappa_\theta \leq \max_{S\in\mathcal A_\theta(G)}
\left(\operatorname{cc}S-|\partial  S|\right).\]

\begin{proof}[Proof of Theorem \ref{thm:lowerbound}]
Set $G = (V,E)$. By Theorem \ref{prop:multicriterion} and Lemma \ref{lem: kappa_lb}, it will suffice for us to demonstrate that $$\max_{S\in\mathcal A_\theta(G)}
\left(\operatorname{cc} S-|\partial  S|\right)\leq |V|\mu_{G_\Gamma}\left(\{\theta\}\right)$$ for any $\theta\in\mathbb{R}$, since 
\(\kappa_\theta\) is no larger than the left hand side. When $\kappa_\theta=0$ the result is immediate, so let us assume to the contrary, and let $S\in\mathcal{A}_\theta$ be a maximising subgraph. Write
\[\mathcal H_\Gamma-\theta I=\begin{pmatrix}
A & B \\
B^* & C
\end{pmatrix}\in M_{V}(\mathcal N\Gamma),\]
where \(A=\mathcal H_\Gamma-\theta I|_{S\times S}\). Since \(S\) is a forest with \(\operatorname{cc} S\) components and each has \(\theta\) as an eigenvalue, \[\dim_\Gamma \mathrm{ker} A \geq \operatorname{cc} S.\]
Also note that \(B\) is zero outside of the rows and columns corresponding to vertices in \(\partial  S\), and so \[\mathrm{rank}_{\Gamma}\,B \leq |\partial  S|.\]

We claim 
\[\dim_\Gamma\mathrm{ker}\,\begin{pmatrix}
A & B \\
B^* & C
\end{pmatrix}\geq \dim_\Gamma \operatorname{ker} A-\mathrm{rank}_\Gamma\,B.\]
From von Neumann nullity \eqref{eq:vonNeumannnull}, this inequality is equivalent to 
\[\mathrm{rank}_\Gamma A+ \mathrm{rank}_\Gamma B +|V|-|S|\geq  \mathrm{rank}_\Gamma \begin{pmatrix}
A & B \\
B^* & C
\end{pmatrix}.\]
Thanks to \eqref{eq:rankineq}, we have 
\[\mathrm{rank}_\Gamma A+ \mathrm{rank}_\Gamma B\geq  \mathrm{rank}_\Gamma \begin{pmatrix}
A & B 
\end{pmatrix},\]
\[\mathrm{rank}_\Gamma \begin{pmatrix}
A & B
\end{pmatrix}+\mathrm{rank}_\Gamma \begin{pmatrix}
B^* & C
\end{pmatrix}\geq \mathrm{rank}_\Gamma \begin{pmatrix}
A & B \\
B^* & C
\end{pmatrix}.\]
The claim then follows from
\[|V|-|S|\geq \mathrm{rank}_\Gamma \begin{pmatrix}
B^* & C
\end{pmatrix}\]
as a consequence of \eqref{eq:vonNeumannnull}. Using this rank inequality, we have 
\(\dim \operatorname{ker} (\mathcal H_\Gamma-\theta I )\geq \operatorname{cc} S- |\partial  S|.\)
\end{proof}

\section{The upper bound}

\label{sec:monotone}

In this section, we will prove the desired upper bound for Theorem \ref{conj}.

\begin{theorem}\label{thm:upperbound}
There exists \(c,\alpha >0\) that depends on \(G_\Gamma\), such that for all closed intervals $I$ of length \(|I| < 1\),
\[
\mu_{G_\Gamma}(I)
\leq \sum_{\theta \in I}
\mu_{G^{\mathrm{ab}}}(\{\theta\})+ \frac{c}{ \log^\alpha (1/|I|)}.
\]
The sum in the right hand side makes sense since \(\mu_{G^{\mathrm{ab}}}(\{\theta\})\) is zero outside of a finite set.
\end{theorem}

We will approach this result in the following way. First, we will compare the spectral measures at atoms using the formula proven in Theorem \ref{prop:multicriterion}. The analysis will split into two cases, whether we can rule out the presence of an atom through the failure of it being a root of the matching polynomial after deletion by a non-empty $2$-regular subgraph, or whether we can only rule it out by seeing that it fails to be a root of the matching polynomial of the graph itself. In the first case, we will use monotone labelling techniques developed in \cite{bordenave2026logarithmic}, and in the latter, we will utilise the matching theory from Section \ref{sec:comb}.

\subsection{Gap in the spectrum for $\theta$-phobic graphs}

To begin, we will show that if \(\theta \in \mathbb R\) is not an \(\ell^2\) eigenvalue of \(G^\ab\), then it is not an \(\ell^2\) eigenvalue of \(G_\Gamma\). Recall that failure to be an $\ell^2$ eigenvalue of $G^\ab$ is equivalent to failure of the graph to be $\mathrm{LMST}(\theta)$. That is, there exists (a possibly empty) \(\gamma\in \mathcal C(G)\) such that \(m_{G\setminus\gamma}(\theta)\neq 0.\) We split into the following possibilities:
\begin{itemize}
    \item \textbf{Case} 1: There exists a non-empty \(\gamma\in \mathcal C(G)\) such that \(m_{G\setminus\gamma}(\theta)\neq 0\);
    \item \textbf{Case} 2: For all non-empty \(\gamma\in \mathcal C(G)\), \(m_{G\setminus\gamma}(\theta)= 0\) while \(m_{G}(\theta)\neq  0\).
\end{itemize}
The latter case motivates the following definition.

\begin{definition}\label{def:hard}
    A graph \(G\) that is not a forest is called \textbf{\(\theta\)-phobic} if \(m_{G}(\theta)\neq 0\) and 
    \[m_{G\setminus\gamma}(\theta)=0 \qquad \forall  \gamma \neq \varnothing, \gamma \in \mathcal C(G),\]
    or equivalently, 
    \[\phi_G(\theta,z)\equiv \mathrm{const.}\neq 0 \qquad \forall z\in \T_G.\]
\end{definition}

Note that the equivalence with $\phi_G(\theta,z)$ being a non-zero constant for all $z\in\T_G$ follows {from the identity \eqref{eq:detmatch} provided in \cite{LiMageeSabriThomas}:} $\phi_G(\theta,z)$ expands as a summation of monomials in the $z_e$ and their inverses whose coefficients are $m_{G\setminus\gamma}(\theta)$ (up to a sign) and so the only surviving term is $m_G(\theta)\neq 0$ which is the coefficient of the monomial that is 1.

The property of being $\theta$-phobic has surprisingly strict spectral consequences for a graph. The goal of this subsection is to prove the following theorem which asserts that if $G$ is $\theta$-phobic then $\theta$ is not in the spectrum of the lifted Schrödinger operator to $G_\Gamma$ but is in the convex hull of its spectrum.

\begin{theorem}\label{thm:surprisethm}
    If \(G\) is \(\theta\)-phobic, then 
    \(\theta \notin \mathrm{Spec}\,G_\Gamma\) for any normal cover $G_\Gamma$ of \(G\). In fact, \(\theta\) is a gap in the spectrum of $G_\Gamma$ in the sense that there exist $\theta_1 <\theta<\theta_2,$  with $\theta_1,\theta_2\in\mathrm{Spec}\,G_\Gamma$. In particular, the spectrum of any normal cover of \(G\) is disconnected.
\end{theorem}

This theorem is one of the two main ingredients in the forthcoming proof of Theorem \ref{thm:upperbound}. It is however of independent interest. Indeed, motivated by solid-states physics applications, gaps in the spectrum of Schrödinger operators have attracted some attention, see notably \cite{zbMATH07207199,zbMATH07750987}.

\subsubsection{A special case: $0$-phobic graphs}
The case of $0$-phobic graphs has a particularly illuminating geometric interpretation which we separate out from general $\theta$. The separation is however unnecessary towards proving Theorem \ref{thm:surprisethm}, and so this can be skipped if it is not of interest to the reader.
In this subsection, we consider the case where $\mathcal{H}$ is the adjacency operator on $G$ and prove the following claim.

\begin{proposition} 
\label{prop:0-hard}
    Suppose that \(G\) has no self-loops and that $\mathcal H$. If \(G\) is \(0\)-phobic for the adjacency operator, then \(0\) is not in the spectrum of the adjacency operator of any (normal) cover of \(G\).
\end{proposition}

\begin{proof}
Since \(G\) is \(0\)-phobic, it has a perfect matching. We show that the perfect matching is unique. Suppose otherwise, and let \(M\) and \(M'\) be two different perfect matchings. Their \textbf{symmetric difference}, i.e., the graph that contains the unshared edges of $M$ and $M'$, is a degree-\(2\) subgraph \(\gamma\) of \(G\). Moreover, the common edges of \(M\) and \(M'\) provides a perfect matching of \(G\setminus\gamma\). Therefore, we have \(m_{G\setminus\gamma}(0)\neq 0\), which contradicts the fact that \(G\) is \(0\)-phobic.

Since \(G\) has a unique perfect matching, it cannot contain an alternating cycle with respect to this matching. Indeed otherwise, the edges in the cycle that are not from the matching, together with edges in the perfect matching that are not transversed by the cycle form a new perfect matching.

Abusing the notation, we write the adjacency matrix of the unique matching perfect matching (as a subgraph of \(G\)) by \(M\). Note that \(M^2=I\). Define \(Q:=A-M\). For any \(r\in \mathbb{N}\), the matrix elements 
\[((MQ)^r)_{u,v}\qquad u,v\in V(G)\]
counts the number of alternating walks from \(u\) to \(v\) of length \(2r\), by first starting through an edge in the matching. Since \(G\)
has a unique perfect matching and no self-loops, all alternating walks have to be alternating paths, that is, no edge is repeated twice. However, since \(G\) is finite and contains no self-loops, there exists \(l\in \mathbb{N}\) such that all paths in \(G\) have length at most \(2l\). Thus \((MQ)^l=0\). Since 
\[MA=I+MQ\]
and \(MQ\) is nilpotent, \(MA\) and thus \(A\) is invertible.

Similarly, we view \(A_\Gamma\) as a matrix 
in \(M_{V(G)}(\mathcal N \Gamma)\). Let \(M_\Gamma\) represent the part of \(A_\Gamma\) that corresponds to the perfect matching, then 
\[M^2_\Gamma=\mathrm{id}.\]
Write \(A_\Gamma=M_{\Gamma}+Q_{\Gamma}\). Again by arguing using the unique perfect matching, \(M_{\Gamma}\,Q_{\Gamma}\) is nilpotent, and \(A_{\Gamma}\) is invertible. Therefore, \(0\notin \mathrm{Spec}\, A_{\Gamma}\). 
 The same proof actually works for any cover \(\tilde{G}\) if one replaces \(M_{V(G)}(\mathcal N \Gamma)\) by \(M_{V(G)}(B(\ell^2(F)))\), where \(F\) is the fibre of the covering map \(\pi:\tilde{G}\rightarrow G \).
\end{proof}

\begin{remark}
    The proof of Proposition \ref{prop:0-hard} actually shows that the existence of a unique perfect matching and no self-loops in $G$ leads to $0$ not being in the spectrum of any (normal) cover. 
\end{remark}

\subsubsection{Proof of Theorem \ref{thm:surprisethm}}

The proof of this theorem is based on the following structural decomposition of \(\theta\)-phobic graphs. For the familiar reader, it plays the role of leaf blocks in the bridge-block tree of odd degree regular graphs used in \cite{LiMageeSabriThomas}. We call a subgraph $H$ of $G$ a \textbf{pending} subgraph if $H$ is connected to $G$ by a bridge edge.

\begin{proposition}\label{prop:LMSTsubgraph}
    If \(G\) is \(\theta\)-phobic, then it 
    contains a pending subgraph satisfying \(\mathrm{LMST}(\theta)\).
\end{proposition}

This proposition actually follows from the following weaker statement.

\begin{lemma}
\label{lem:pending}
    If \(G\) is \(\theta\)-phobic, then it has a pending subgraph whose matching polynomial has $\theta$ as a root.
\end{lemma}

Before proving Lemma \ref{lem:pending}, we start with a useful property of $\theta$-phobic graphs.

\begin{corollary}[Corollary of Lemma \ref{lem: path lemma}]
\label{lemma:existenceogpositivevertex}
    If \(G\) is \(\theta\)-phobic, then it contains a \(\theta\)-positive vertex.
\end{corollary}

\begin{proof}
    Since \(G\) is \(\theta\)-phobic, any path \(P\) obtained by ignoring one edge of any cycle of \(G\) satisfies
    \(m_{G\setminus P}(\theta)=0.\)
    By Lemma \ref{lem: path lemma}, since $G$ cannot contain a cycle \(\gamma\) with \(m_{G\setminus\gamma}(\theta)\neq 0\), there must exist a vertex \(v\) in \(G\) such that \(m_{G\setminus\{v\}}(\theta)=0\).
\end{proof}

We are now ready for the proof of Lemma \ref{lem:pending}.

\begin{proof}[Proof of Lemma \ref{lem:pending}]
    By Corollary \ref{lemma:existenceogpositivevertex}, we let \(v\) be a \(\theta\)-positive vertex of \(G\) so that \(\mathrm{mult}_{\theta}(G\setminus\{v\})=1\). By Proposition \ref{prop:mult-roots}, we then have
    \[\operatorname{cc}  D_{\theta}(G\setminus\{v\})-\left|\partial  D_{\theta}(G\setminus\{v\})\right|=1.\]
    Let \(K_1,\cdots,K_q\) be the connected components of \( D_{\theta}(G\setminus\{v\})\), then we have \(|\partial D_{\theta}(G\setminus\{v\})|=q-1\). Define 
    \[X=\{v\}\sqcup\,\partial   D_{\theta}(G\setminus\{v\}).\]
    Note that \(|X|=q\). We may define a bipartite graph \(W\) with vertex set \(X\sqcup Y\) where \(Y=\{K_1,\cdots K_q\}\): for any edge between some vertex in $X$ and some vertex in $K_i\in Y$ in $G$, we have a corresponding edge in $W$.
    
    Note that $v$ does connect to at least one component in $Y$ 
    since otherwise $\theta$ would be a root of the matching polynomial of $G$. Indeed, $v$ is $\theta$-positive in $G$, and so $D_\theta(G)=D_\theta(G\setminus\{v\})$. Moreover, $\partial  D_\theta(G)$ contains all of $\partial  D_\theta(G\setminus\{v\})$ and nothing more since the only other possible vertex in this boundary would be $v$, which is ruled out if $v$ does not connect to a component of $Y$. But then Proposition \ref{prop:mult-roots} would imply $\theta$ is a root of $m_G(x)$ which is not the case.
    
    As in the proof of Proposition \ref{prop:0-hard}, we will now show that $\theta$-phobic implies the existence of a unique perfect matching in $W$. First of all, a perfect matching always exists since after the removal of \(v\), the remaining bipartite graph is precisely $G_\theta$ which has a positive surplus from \(X\setminus\{v\}\) by Lemma \ref{lemma:Halltype}. This means that by Theorem \ref{thm:Halls_thm} if we pick a neighbour $K_i$ of $v$ in $W$ then we can find a matching in $W$ that contains $X\setminus\{v\}$ and avoids $K_i$. Adding an edge in $W$ from $v$ to $K_i$ then provides a perfect matching.
    
    If the perfect matching is not unique, the symmetric difference of two perfect matchings (as in Proposition \ref{prop:0-hard}) is a degree-\(2\) subgraph, say \(\gamma_W\), of \(W\). If this degree-\(2\) graph doesn't contains all vertices in \(W\), then the common edges of the two matchings contains all vertices that are left. We show that the existence of this decomposition leads to a contradiction with $\theta$-phobic.
    
    Without loss of generality, we assume that \(\gamma_W\) is connected and visits 
    \[u_1,K_{i_1}, u_2, K_{i_2} , \ldots, u_l, K_{i_l} ,u_1:=u_{l+1}\]
    in order. Denote by \(q^-_j,q^+_j\) the vertices in \(K_{i_j}\) that are connected to \(u_j\) and \(u_{j+1}\) in \(G\) corresponding to the edges in this \(\gamma_W\) that connects \(u_j\) and \(u_{j+1}\) to \(K_{i_j}\). By Lemma \ref{lemma:pathincritical}, there exists a (directed) path \(P_j\) in \(K_{i_j}\) connecting \(q^-_j\) and \(q^+_j\), such that 
    \(m_{K_{i_j}\setminus P_j}(\theta)\neq 0\). If $q^-_j=q^+_j$ then this path is just this vertex. From \(\gamma_W\) and \(P_j\), we then construct 
    a cycle \(\gamma_G\) in \(G\) that transverses in order
    \[u_1, P_1, u_2, P_2, \ldots, u_l, P_l, u_1.\]

    We claim that \(m_{G\setminus \gamma_G}(\theta)\neq 0\), leading to a contradiction of \(\theta\)-phobic. If \(\gamma_W\) contains all vertices of $W$, then each connected component of what is left in $G\setminus \gamma_G$ is either of the form \(K_{i_j}\setminus P_j\), or the connected components of  \[G\setminus \left(\{v\}\sqcup D_{\theta}\left(G\setminus \{v\}\right)\sqcup\partial D_{\theta} \left(G\setminus \{v\}\right)\right).\]
    The matching polynomial of all these graphs are non-vanishing at \(\theta\), and thus since the matching polynomial of a disconnected graph is the product of the matching polynomials of the components, we have \(m_{G\setminus\gamma_G}(\theta)\neq 0\) which is a contradiction to $\theta$-phobic.

    If \(\gamma_W\) does not contain all vertices in $W$, then \(l<q\). Suppose that \(u_{l+1},\ldots,u_q\) and \(K_{i_{l+1}},\ldots, K_{i_{q}}\) are the vertices not in \(\gamma_W\) in \(W\). Denote by \(F\) the collection of edges in \(G\) between the vertices $u_{l+1},\ldots,u_q$ and the vertices in the $K_{i_{l+1}},\ldots,K_{i_q}$. By \eqref{eq:multiedgerecursion} we have
    \begin{equation*}
       m_{G\setminus \gamma_G}(\theta)=\sum_{M \in \mathcal{M}_{G\setminus\gamma_G},\; E(M) \subset F} (-1)^{|E(M)|}\prod_{e\in \vec{E}(M)}w_e\cdot m_{G\setminus(\gamma_G\cup F\cup V(M))}(\theta).
    \end{equation*}
    Suppose that one of the matchings $M$ on the right hand side has no intersection with some \(K_{i_j}\), \(j\in \{l+1,\ldots,q\}\), then \(K_{i_j}\) is a connected component of \(G\setminus(\gamma_G\cup F\cup V(M))\), thus
    \[m_{G\setminus(\gamma_G\cup F\cup V(M))}(\theta)=0.\] Therefore, we only need to consider those \(M\) that intersect with all of \(K_{i_{l+1}},\ldots, K_{i_{q}}\). In this case \(V(M)\) must contain \(u_{l+1},\ldots,u_q\) since its edges are contained in $F$ which have one of their endpoints among these vertices. Moreover this implies that \(|E(M)|=q-l\). Such a matching \(M\) always exists because the common edges of the two perfect matchings we are considering is an example. 

    In this case, we have 
    \begin{equation*}
       m_{G\setminus \gamma_G}(\theta)=(-1)^{q-l} \sum_{M \in \mathcal{M}_{G\setminus\gamma_G},\;|E(M)|=q-l,\; E(M) \subset F} \prod_{e\in \vec{E}(M)}w_e\cdot m_{G\setminus(\gamma_G\cup F\cup V(M))}(\theta).
    \end{equation*}
    Each component of such \(G\setminus(\gamma_G\cup F\cup V(M))\) is either \(K_{i_j}\setminus P_j\) for \(j\leq l\), or the connected components of \(G\setminus\left(\{v\}\cup D_{\theta}\left(G\setminus\{v\}\right)\cup\partial D_\theta  \left(G\setminus\{v\}\right)\right)\), or \(K_{i_{j}}\) for \(l<j\leq q\), after deleting a vertex. The matching polynomials of these graphs are non-vanishing at \(\theta\) and thus \(m_{G\setminus \gamma_G}(\theta)\neq 0\), a contradiction to $\theta$-phobic.

    So, we have that \(W\) has a unique perfect matching and is bipartite. We claim that one of the \(K_i\) has to be a leaf. If not, we consider the alternating walk starting from any \(v\in X\) and the matching edge covering it. Such a walk is always non-backtracking since it is alternating. Furthermore, it can never visit a vertex twice because that gives rise to an alternating cycle which contradicts the uniqueness of the perfect matching (we can interchange the edges in the cycle that are in the matching with those that aren't to get a new perfect matching). 
    
    The walk cannot terminate at a $K_i$ otherwise it is a leaf, since the incoming edge to it is in the matching and so we can then traverse any other outgoing edge if one existed so the path wouldn't terminate. This means it must terminate at a vertex in $X$ but this cannot happen. Indeed, an incoming edge to a vertex in $X$ is not in the matching by construction. However, the matching is perfect, so there is an edge in the matching containing that vertex which we can then traverse to extend the path. It follows that one of the $K_i$ must indeed be a leaf in \(W\) and this means it is a pending subgraph in \(G\).
\end{proof}

\begin{proof}[Proof of Proposition \ref{prop:LMSTsubgraph}]
We prove the result by induction on the number of vertices in $G$. Precisely, we induct on the statement: if $G$ is a graph on $k$ vertices that is $\theta$-phobic with respect to some $\mathcal{H}$, then it has a pending subgraph that satisfies $\mathrm{LMST}(\theta)$ with respect to {(the restriction of)} $\mathcal{H}$.

Suppose first that \(|V(G)|=1\), then the statement is vacuously true, because there are no $\theta$-phobic graphs on one vertex for any $\mathcal{H}$. Indeed, $\theta$-phobic means $G$ must be a bouquet but then deleting any loop gives the empty graph which has matching polynomial equivalent to 1, and so $\theta$ is never a root.

Assume that the statement is true whenever \(|V(G)|\leq k\), we show it also holds for \(|V(G)|= k+1\), that is, we wish to prove that, given $\mathcal{H}$ on $G$, there exists a pending subgraph $Q$ of $G$ that satisfies $\phi_Q(\theta,z)\equiv 0$ for all $z\in \T_Q$. 

Let \(K\) be a pending subgraph of \(G\) that connects to the rest of the graph \(L=G\setminus K\) by a (directed) edge \(e\) as guaranteed by Lemma \ref{lem:pending}. Let $u\in K$ and $v\in L$ be the endpoints of $e$. Furthermore, for $z\in \T_G$ and any subgraph $G'$ of $G$, let $z_{G'}$ be the coordinates of $z$ restricted to the edges that appear in $G'$.  Lemma \ref{lemma:turnonoff} states that because $G$ is $\theta$-phobic,
    \begin{align}
    \label{eq:pending-decomp}\phi_K(\theta,z_K) \phi_L(\theta,z_L) -|w_e|^2\phi_{K\setminus\{u\}}(\theta,z_{K\setminus\{u\}})\phi_{L\setminus\{v\}}(\theta,z_{L\setminus\{v\}})\equiv c\neq 0.\end{align}
 Note that we have \(m_K(\theta)=0\) and  \(m_{K\setminus\{u\}}(\theta)\neq 0 \). The latter claim follows from integrating \eqref{eq:pending-decomp} over the torus \(\T_K\) to obtain

\begin{align*}
    0\neq c \equiv&\phi_L(\theta,z_L)\int_{\mathbb T_K}\phi_K(\theta,z_K)-|w_e|^2 \phi_{L\setminus\{v\}}(\theta,z_{L\setminus\{v\}}) \int_{\mathbb T_K}\phi_{K\setminus\{u\}}(\theta,z_{K\setminus\{u\}})\\
    =&0-|w_e|^2 \phi_{L\setminus\{v\}}(\theta,z_{L\setminus\{v\}}) m_{K\setminus\{u\}}(\theta),
\end{align*}
where we have used that the integral of the characteristic polynomial over all of its $z$ variables is just the matching polynomial. From this equality we also see that $\phi_{L\setminus\{v\}}(\theta,z_{L\setminus\{v\}})$ is constant in $z$ because
    $$\phi_{L\setminus\{v\}}(\theta,z_{L\setminus\{v\}}) = \frac{-c}{|w_e|^2 m_{K\setminus\{u\}}(\theta)}\equiv c'\neq 0.$$

Now, if \(\phi_K(\theta,z_k)\equiv 0\) or \(\phi_L(\theta,z_L)\equiv 0\) for fixed $\theta$, then we are done, $K$ or $L$ would be a desired pending subgraph. So let us suppose otherwise. In this case, $K$ is not a forest because then $\phi_K(\theta,z_K)=m_K(\theta)=0$. So, there exists $\tilde{z}_k\in\T_K$ for which $\phi_K(\theta,\tilde{z}_K)\neq 0$ and thus for any $z_L\in\T_L$ we have from \eqref{eq:pending-decomp} that
\[\phi_L(\theta,z_L)=\frac{c'|w_e|^2\phi_{K\setminus\{u\}}(\theta,\tilde{z}_{K\setminus\{u\}})+c}{\phi_K(\theta,\tilde{z}_K)}\equiv A,\]
which is constant for all $z_L\in\T_L$ and non-zero because otherwise $\phi_L(\theta,z_L)\equiv 0$.

Now consider \eqref{eq:pending-decomp} again, and this time divide through by $A$ to obtain for any $z_K\in\T_K$
\[\phi_K(\theta,z_K)-|w_e|^2\frac{c'}{A}\phi_{K\setminus\{u\}}(\theta,z_{K\setminus\{u\}})\equiv \frac{c}{A}\neq 0.\]

Define \(\mathcal{H}'\) to be the Schr\"odinger matrix on \(G\) given by shifting the potential at \(u\) by \(|w_e|^2\frac{c'}{A}\) and leaving other potentials and weights unchanged. That is,
    $$\mathcal{H}'(z) = \mathcal{H}(z)+\mathrm{diag}\left(0,\ldots,0,\frac{c'}{A}|w_e|^2,0,\ldots,0\right), $$
where the non-zero entry of the diagonal matrix is in the position corresponding to $u$.
We now expand the characteristic polynomial.  We obtain for any $z_K\in\T_K$ that 
    \begin{align*}
        \phi_K^{\mathcal{H}'}(\theta,z_K) &= \det\left(\left[\theta I -\mathcal{H}(z)-\mathrm{diag}\left(0,\ldots,0,\frac{c'}{A}|w_e|^2,0,\ldots,0\right)\right]_K\right)\\
        &=\phi_K(\theta,z_K)-\frac{c'}{A}|w_e|^2\det\left(\left[\theta I-\mathcal{H}(z)\right]_{K\setminus\{u\}}\right)\\
        &=\phi_K(\theta,z_K)-\frac{c'}{A}|w_e|^2\phi_{K\setminus\{u\}}(\theta,z_{K\setminus\{u\}})\\
        &\equiv \frac{c}{A}\neq 0.
    \end{align*} 
Thus, since $K$ is not a forest, it is \(\theta\)-phobic for $\mathcal{H}'$ and so by induction, there exists a pending subgraph $Q$ of $K$ that satisfies $\mathrm{LMST}(\theta)$ with respect to $\mathcal{H}'$. 

If  \(u\notin V(Q)\), then \(\mathcal{H}|_Q=\mathcal{H}'|_Q\), and so \(Q\) is the desired subgraph of $G$. Else suppose that \(u\in V(Q)\), then 
\[\phi^{\mathcal{H}'}_Q(\theta,z_Q)\equiv 0,\]
for any $z_Q\in\T_Q$ since $Q$ is $\mathrm{LMST}(\theta)$ with respect to $\mathcal{H}'$. But then, for any $z\in\T_G$ we have 
\[0\equiv \phi_Q^{\mathcal{H}'}(\theta,z_Q)=\phi_Q(\theta,z_Q)-|w_e|^2\frac{c'}{A}\phi_{Q\setminus\{u\}}(\theta,z_{Q\setminus\{u\}}).\] 
Since $A=\phi_L(\theta,z_L)$ and $c'=\phi_{L\setminus\{v\}}(\theta,z_{L\setminus\{v\}})$ (which are constant for any choice of $z_G$) {and $\{u,v\}$ is a bridge of $L\sqcup Q$}, we see that this is equivalent by Lemma \ref{lemma:turnonoff} to 
    $$\phi_{L\sqcup Q}(\theta,z_{L\sqcup Q})=0.$$
Since $L\sqcup Q$ attaches to $K\setminus Q$ {by the bridge edge between $Q$ and $K\setminus Q$}, it is the desired pending subgraph of $G$ satisfying $\mathrm{LMST}(\theta)$ with respect to $\mathcal{H}$.
\end{proof}

We are ready for the proof of Theorem \ref{thm:surprisethm}.

\begin{proof}[Proof of Theorem \ref{thm:surprisethm}]
We show that for any finite \(G\) such that \(\phi_G(\theta, z_G)\equiv c\neq 0\), \(\theta \)
is not in the spectrum of \(G_{\Gamma}\) but in its convex hull.
We proceed by induction on \(|V(G)|\). The base case when \(|V(G)|=1\) is vacuously true because then $G$ is not $\theta$-phobic. So, assume that the statement holds when \(|V(G)|\leq k\), we prove the case when \(|V(G)|=k+1\). Let $K$ be the pending subgraph for $G$ obtained from Proposition \ref{prop:LMSTsubgraph}, $L$ its complement {and $e = (v,u)$ the bridge edge, $u \in K$, $v \in L$}. Viewing \(\mathcal H_\Gamma-\theta I\) as a matrix in \(M_{V(G)}(\mathbb C\Gamma)\subset M_{V(G)}(\mathcal N\Gamma)\), we write
\[\mathcal H_\Gamma-\theta I=\USSEnterprise
  {{B_1}{C_1}{0}{0}}
  {{C_1^*}{M_1}{{ T}}{0}}
  {{0}{{T^*}}{M_2}{C_2^*}}
  {{0}{0}{C_2}{B_2}}.\]
Here the submatrices \[\begin{pmatrix}
B_1 & C_1 \\
C_1^* & M_1
\end{pmatrix} \in M_{V(L)} (\mathbb C \Gamma), \quad \begin{pmatrix}
M_2 & C_2^* \\
C_2 & B_2
\end{pmatrix} \in  M_{V(K)} (\mathbb C \Gamma) \] correspond to the restriction of
\(\mathcal H_\Gamma-\theta I\) on \(V(L)\) and \(V(K)\) respectively. The entries \(M_1,M_2\) in $\mathbb C \Gamma$ correspond to the restriction on \(\{v\}\) and \(\{u\}\). The two extra entries {\(T = w_e \rho_{\Gamma} (g_e) \)  and \(T^* = w_{\bar e} \rho_{\Gamma} (g_{\bar e}) \)}  in the block decomposition correspond to the bridge between \(K\) and \(L\).

Recall that \(B_1\) is the lift of the operator \(\mathcal H|_{L\setminus\{v\}}-\theta I\) which acts on \(\pi^{-1}_\Gamma(L\setminus\{v\})\), where $\pi_\Gamma:G_\Gamma\to G$ is the covering map. Identically to the proof of Proposition \ref{prop:LMSTsubgraph}, we obtain that for any $z_{L\setminus\{v\}}\in\T_{L\setminus\{v\}}$, we have $\phi_{L\setminus\{v\}}(\theta,z_{L\setminus\{v\}})\equiv c\neq 0$. If $L\setminus\{v\}$ is a tree, then $B_1$ is equivalent to $\mathcal{H}(z)-\theta I$ restricted to $L\setminus\{v\}$ and is hence invertible due to the non-zero determinant. Otherwise, $L\setminus\{v\}$ is $\theta$-phobic and so we can apply the induction statement to obtain invertibility of $B_1$ in \(M_{|V(L)|-1}\left(\mathcal{N}\Gamma\right)\). 

We similarly obtain invertibility for $B_2$ from the following facts: $G$ is $\theta$-phobic, so that $\phi_G(\theta,z_G)\equiv c'\neq 0$ and $K$ satisfies $\mathrm{LMST}(\theta)$ so that $\phi_K(\theta,z_K)\equiv 0$ giving by Lemma \ref{lemma:turnonoff},
    $$0\neq c' \equiv -|w_e|^2\phi_{K\setminus\{u\}}(\theta,z_{K\setminus\{u\}})\phi_{L\setminus\{v\}}(\theta,z_{L\setminus\{u\}}),$$
which implies that $\phi_{K\setminus\{u\}}(\theta,z_{K\setminus\{u\}})\equiv A\neq 0$ for all $z_{K\setminus\{u\}}\in\T_{K\setminus\{u\}}$. So, either $K\setminus\{u\}$ is a tree and $B_2$ is invertible, or it is $\theta$-phobic and so the induction statement gives invertibility of $B_2$.

Note that by Definition \ref{def:dos}, \eqref{eq:vonNeumannnull} and Theorem \ref{thm:lowerbound}, \[1\leq |V(K)|\mu_{K^\ab}(\{\theta\}) \leq  |V(K)|\mu_{K_\Gamma}(\{\theta\})=|V(K)|-\mathrm{rank }_\Gamma\begin{pmatrix}
M_2 & C_2^* \\
C_2 & B_2
\end{pmatrix},\]
and so we have by Lemma \ref{lem:rank-identities},
\[
\begin{aligned}
    |V(K)|-1&\geq \mathrm{rank }_\Gamma\begin{pmatrix}
M_2 & C_2^* \\
C_2 & B_2
\end{pmatrix}=\mathrm{rank }_\Gamma\begin{pmatrix}
M_2-C_2^*B_2^{-1}C_2 & 0 \\
C_2 & B_2
\end{pmatrix}\\&\geq \mathrm{rank }_\Gamma \left(M_2-C_2^*B_2^{-1}C_2\right)+\mathrm{rank }_\Gamma B_2\\&=\mathrm{rank }_\Gamma \left(M_2-C_2^*B_2^{-1}C_2\right)+|V(K)|-1,
\end{aligned}
\]
therefore \(\mathrm{rank }_\Gamma \left(M_2-C_2^*B_2^{-1}C_2\right)=0\) and thus \(M_2=C_2^*B_2^{-1}C_2\) {since  the tracial state is faithful in $\mathcal N \Gamma$}. 

By performing Gaussian elimination repeatedly and using $T T^* = |w_e|^2 I$, we have 
 { \[
    \USSEnterprise
  {{B_1}{C_1}{0}{0}}
  {{C_1^*}{M_1}{T}{0}}
  {{0}{T^*}{M_2}{C_2^*}}
  {{0}{0}{C_2}{B_2}}\xrightarrow{ X_1 \times} \USSEnterprise
  {{B_1}{C_1}{0}{0}}
  {{C_1^*}{M_1}{T}{0}}
  {{0}{T^*}{0}{0}}
  {{0}{0}{C_2}{B_2}}\xrightarrow{ X_2 \times} \USSEnterprise
  {{B_1}{0}{0}{0}}
  {{C_1^*}{0}{T}{0}}
  {{0}{T^*}{0}{0}}
  {{0}{0}{C_2}{B_2}},\]
  where \[X_1=\USSEnterpriseD
  {{I}{0}{0}{0}}
  {{0}{I}{0}{0}}
  {{0}{0}{I}{Y_1}}
  {{0}{0}{0}{I}} \qquad \text{and}  \qquad X_2=\USSEnterpriseD
  {{I}{0}{Y_2}{0}}
  {{0}{I}{Y_3}{0}}
  {{0}{0}{I}{0}}
  {{0}{0}{0}{I}}\]}
with \(Y_1=C_2^*B_2^{-1}\), \(Y_2=-\frac{C_1 T}{|w_e|^2}\) and \(Y_3=-\frac{M_1 T}{|w_e|^2}\)
and \(X_i \times\) means to do the multiplication on the left. Since the last matrix is lower block-triangular { with invertible blocks on the block diagonal}, it is invertible. For the same reason, \(X_1\) and \(X_2\) are invertible too.
  Thus \(\mathcal H_\Gamma-\theta I\) is invertible and hence $\theta\notin\mathrm{Spec}\, G_\Gamma$.

Now we show \(\theta\) is in fact in a spectral gap. Let \(G_\Gamma\) be a normal cover of \(G\)
viewing as a derived voltage graph. By \(\theta\)-phobic, we have \(\mathrm{mult}_\theta (G \setminus \gamma)>0\) for any non-empty \(\gamma\in \mathcal{C}(G)\). Note that since \(G\) is by definition not a tree, there always exists a non-empty \(2\)-regular subgraph.
Theorem \ref{thm:lowerbound} then shows that 
\(\theta\) is an eigenvalue of \((G\setminus \gamma)_\Gamma\) for all non-empty \(\gamma\in \mathcal{C}(G)\), therefore we have 
\[\min\,\operatorname{Spec}\, (G\setminus \gamma)_\Gamma\leq \theta \leq  \max\,\operatorname{Spec} (G\setminus \gamma)_\Gamma.\]
Note that \((G \setminus \gamma)_\Gamma\) is an induced subgraph of \(G_\Gamma\), by the min-max theorem, we have \[\min\,\operatorname{Spec}\, G_\Gamma \leq \min\,\operatorname{Spec}\, (G\setminus \gamma)_\Gamma\leq \theta \leq  \max\,\operatorname{Spec} (G\setminus \gamma)_\Gamma \leq \max\,\operatorname{Spec}\, G_\Gamma.\] 
Since \(\theta\) is not in the spectrum, 
we have 
\[\min\,\operatorname{Spec}\, G_\Gamma < \theta < \max\,\operatorname{Spec}\, G_\Gamma.\]
 It concludes the proof.
\end{proof}

\subsection{Proof of Theorem \ref{thm:upperbound}}

\subsubsection{Monotone labelling revisited}

We first outline the various notions of the monotone labelling method first introduced in \cite{Bo.Se.Vi2017} and subsequently generalised by the first named author in \cite{bordenave2026logarithmic}.

\begin{definition}
    Let \(G=(V,E)\) be a graph. A map \(\eta: V\rightarrow \Z\) is called a labelling of the vertices of \(G\) with integers. We shall call a vertex \(v\)
\begin{itemize}
    \item prodigy if it has a neighbour \(w\) with \(\eta(w) < \eta (v)\) so that all other neighbours of \(w\) also have
label less than \(\eta(v)\);
     \item level if not prodigy and if all of its neighbors have the same or lower labels;
     \item bad if none of the above holds.
\end{itemize}
\end{definition}

\begin{definition}
    On a voltage graph \(G_\Gamma\), a labelling is called periodic if for any \(u\in V_{G_\Gamma}\), \(\eta(u)=\eta(\Gamma\cdot u)\). 
\end{definition}

Note that with respect to a fixed periodic labelling, \(u\) is prodigy, level or bad if and only if \(\Gamma\cdot u \) is prodigy, level or bad. Thus we can always find induced subgraphs \(P\), \(B\) and \(L\) of \(G\) such that \(P_\Gamma\), \(B_\Gamma\) and \(L_\Gamma\) contain 
prodigy, level or bad vertices of \(G_\Gamma\), respectively. The following result is a minor variant of a result of the first named author \cite[Lemma 1]{bordenave2026logarithmic} whose proof follows by replacing finite dimensions with von Neumann dimensions (it is also a corollary of \cite[Theorem 4]{bordenave2026logarithmic}). This is made by possible by the periodicity of the labelling.

\begin{proposition}[{\cite[Lemma 1]{bordenave2026logarithmic}}]\label{th:monof}
Fix a finite graph \(G=(V,E)\) with a Schr\"odinger operator \(\mathcal{H}\). 
Let $\eta$ be a periodic labelling on $G_\Gamma$ with $n$ distinct values. There exists $\beta = \beta (G, \mathcal{H}) \geq 1$ such that for any interval $I$ of length at most $\beta^{-n}$, we have  
$$
\mu_{G_\Gamma}(I) \leq  \frac{ |B|}{|V|}+ \frac{1}{|V|} \sum_k  |L^k| \mu_{L^k_\Gamma}(I_n) ,
$$
where,  \(B\) and \(L^k\) are induced subgraphs of \(G\) such that vertices in \(B_\Gamma\) and \(L^k_\Gamma\) are bad and level vertices with label $k$, respectively and $I_n $ is the closed interval of length $|I|^{1/ (6 n  \ln n)}$ with the same centre. 
Moreover, if there are no edges between level vertices with different label, we have 
$$
\mu_{G_\Gamma}(I) \leq  \frac{ |B|}{|V|}+ \frac{|L|}{|V|} \mu_{L_\Gamma}(I_n) ,
$$
where \(L\) is the induced subgraph of \(G\) such that the vertices in \(L_\Gamma\) are the level vertices.
\end{proposition}

We now introduce a height function on \(G_\Gamma\) whenever the covering map to \(G\) factorises through \(G^\ab\). The construction is to first define a height function on the universal cover and then reduce it to one on \(G_\Gamma\).

Let \(\gamma=(V_\gamma,E_\gamma)\) be a degree-\(2\) subgraph of the base graph. Fix an orientation on each of its connected component. Define \(\phi^\gamma: \vec{E}(G) \rightarrow \mathbb{Z}\) by

\begin{align}
\label{eq:phi_gamma}
\phi^\gamma(e)
=
\begin{cases}
1, & \text{\(e\in \vec{E}(\gamma)\) agrees with the orientation};\\
-1, & \text{\(e\in \vec{E}(\gamma)\) disagrees with the orientation};\\
0, & \text{otherwise}.
\end{cases}
\end{align}

Fix \(o\in V(G)\). We view the vertices of the universal cover \(G^{\uni}\) as non-backtracking walks \(\{e_1,\cdots,e_n\}\) based at $o$. We define a height function on \(G^{\uni}\) by 
\[h^\gamma\left(\{e_1,\cdots,e_n\}\right)=\sum_{i=1}^n \phi^\gamma(e_i).\]
If the starting vertex of \(\{e_{1'},\cdots,e_{n'}\}\) is the same as the end vertex of \(\{e_1,\cdots,e_n\}\), denote by 
\[\{e_1,\cdots,e_n\}\cdot \{e'_{1},\cdots,e'_{n'}\}\]
the non-backtracking walk that the walk \(\{e_1,\cdots,e_n,e'_{1},\cdots,e'_{n'}\}\) reduces to. Note that we always have 
\begin{equation}\label{eq:heighthomo}
    h^\gamma\left(\{e_1,\cdots,e_n\}\cdot \{e'_{1},\cdots,e'_{n'}\}\right)=\sum_{i=1}^n \phi^\gamma(e_i)+\sum_{i=1}^{n'} \phi^\gamma(e'_i).
\end{equation}

Recall that there is an isomorphism between \(\pi_1(G)\) and the collection of closed non-backtracking walks based at \(o\). By restricting to closed walks at \(o\), \(h^\gamma\) induces a map \(\mathbf{h}^\gamma: \pi_1(G)\rightarrow \mathbb Z\). Thanks to \eqref{eq:heighthomo}, \(\mathbf {h}^\gamma\) is a group homomorphism. Moreover, for any element \(\gamma_0\) in the commutator group \([\pi_1(G),\pi_1(G)]\), \(h^\gamma(\gamma_0)=0\), since each undirected edge is traversed in \(\gamma_0\) from both directions for the same number of times.

We view the fundamental group \(\pi_1(G_\Gamma)\) as a subgroup of \(\pi_1(G)\), and since \(G_{\Gamma}\) covers \(G{\mathrm{ab}}\), \(\pi_1(G_\Gamma) \leq [\pi_1(G),\pi_1(G)]\). Note that \(\Gamma \cong \pi_1(G)/\pi_1(G_\Gamma)\). Vertices in \(G_\Gamma\) can be seen as equivalence classes of non-backtracking walks starting from \(o\) under the relation  \(\{e_{1},\cdots,e_{n}\}\sim \{e_{1'},\cdots,e_{n'}\}\) if and only if there exists \(\gamma_0\in \pi_1(G_\Gamma)\), \[\{e_{1},\cdots,e_{n}\}=\gamma_0\cdot \{e_{1'},\cdots,e_{n'}\}.\] 
Since \(h(\gamma_0)=0\), we have
\(h^\gamma(\{e_{1},\cdots,e_{n}\})=h^\gamma\left(\{e_{1'},\cdots,e_{n'}\}\right).\) Therefore, the height function \(h^\gamma\) on \(G^\uni\) gives a well-defined height function on \(G_\Gamma\).

\subsubsection{Proof of Theorem \ref{thm:upperbound}}
Now we are ready to prove Theorem \ref{thm:upperbound}. 

\begin{proof}[Proof of Theorem \ref{thm:upperbound}]
We argue by induction on the number of edges of the base graph. If the graph has no edges, then it is a singleton, and the theorem is trivially true. Suppose the theorem holds for all graphs on at most \(k\) edges, and fix a graph \(G\) with \(|E(G)|=k+1\).

Define
\[S_0=\{\theta\in \R: \text{there exists \(\gamma\in \mathcal C(G)\) such that \(\mathrm{mult}_\theta (G\setminus \gamma)>0\)} \}.\]
From Theorem \ref{prop:multicriterion}, it suffices to show that we can find \(c,\alpha, \varepsilon>0\) such that for any closed interval \(I\) with \(|I|\leq \varepsilon\), 
\begin{equation}
    \mu_{G_\Gamma}(I)\leq \frac{1}{|V(G)|}\sum_{\theta\in S_0\cap I} \min_{\gamma\in \mathcal C(G)} \mathrm{mult}_\theta (G\setminus \gamma)+\frac{c}{\log^\alpha (1/|I|)}.
\end{equation}
Note that \(|S_0|<\infty\), thus there exists \(\delta>0\) such that for any interval \(|I|\leq \delta\), \(I\) contains at most one element in \(S_0\). Define the following two subsets \(S_1\) and \(S_2\) of \(S\) by
\[S_1=\{\theta\in S_0: \text{\(\varnothing\) is the unique minimizer of \(\mathrm{mult}_\theta (G\setminus \gamma)\) and \(\mathrm{mult}_\theta (G)=0\)} \},\]
\[S_2=\{\theta\in S_0: \text{\(\varnothing\) is the unique minimizer of \(\mathrm{mult}_\theta (G\setminus \gamma)\) and \(\mathrm{mult}_\theta (G)>0\)} \}.\]
We have the following three estimates, related to \(S_1\), \(S_2\) and any nonempty \(2\)-regular subgraph \(\varnothing \neq \gamma \in \mathcal{C}(G)\), respectively.
\begin{itemize}
    \item \textbf{Estimate I.} For any \(\theta_1\in S_1\), \(G\) is \(\theta_1\)-phobic by definition. By Theorem \ref{thm:surprisethm}, \(\theta_1\) is not in the spectrum and there exists \(\epsilon_1(\theta_1)>0\) such that for any \(0<\epsilon<\epsilon_1(\theta_1)\),
\[\mu_{G_\Gamma}([\theta_1-\epsilon,\theta_1+\epsilon])=0.\]

\item \textbf{Estimate II.} For any \(\theta_2\in S_2\), write \(p=\mathrm{mult}_{\theta_2} (G)\). In this case, by Lemma \ref{lm:essenexistence}, we can find \(v_1,\ldots, v_p\in V(G)\), such that \(m_{G\setminus\{v_1,\ldots,v_p\}}(\theta_2)\neq 0\). Then \(\theta\) is not an eigenvalue of the maximal abelian cover of any connected component of \(G\setminus\{v_1,\ldots,v_p\}\). Let \(C^1,\ldots,C^l\) be the connected components of \(G\setminus\{v_1,\ldots,v_p\}\). Then \((G\setminus\{v_1,\ldots,v_p\})_\Gamma\) is the disjoint union of \(C^i_\Gamma\). By the definition of the spectral measure, \[|V(G\setminus\{v_1,\ldots,v_p\})|\cdot \mu_{(G\setminus\{v_1,\ldots,v_p\})_\Gamma}=\sum_{i}|V(C_i)|\cdot \mu_{C^i_\Gamma}.\]
By Lemma \ref{lemma:Hughesfree}, $C_\Gamma^i$ covers the maximal abelian cover of \(C^i\). Since each $C^i$ has less edges than $G$, the induction hypothesis provides the existence of \(c, \alpha, \epsilon_2(\theta_2)>0\), such that for each \(1\leq i\leq l\) and \(0<\epsilon \leq \epsilon_2(\theta_2)\), there is no eigenvalues of \((C^i)^\ab\)
in \([\theta_2-\epsilon,\theta_2+\epsilon]\) and 
$$
\mu_{C_\Gamma^i}([\theta_2-\epsilon,\theta_2+\epsilon]) \leq \frac{c}{\log ^{\alpha}\epsilon^{-1}}.
$$
By the Cauchy interlacing of Lemma \ref{lemma:Cauchy interlacing} we then obtain that for some \(c'>0\), we have
\[
\begin{aligned}
    |V(G)|\mu_{G_\Gamma} ([ \theta_2-\epsilon,\theta_2+\epsilon ])
    \leq &\; \mathrm{mult}_{\theta_2} (G)+\left(|V(G)|-\mathrm{mult}_{\theta_2} (G)\right)\mu_{ (G\setminus\{v_1,\ldots,v_p\})_\Gamma}([\theta_2-\epsilon,\theta_2+\epsilon ])\\ \leq &\;  \mathrm{mult}_{\theta_2} (G) + \frac{c'}{\log ^{\alpha}\epsilon^{-1}}.
\end{aligned}
\]

\item \textbf{Estimate III.} For any \(\varnothing \neq \gamma\in \mathcal{C}(G)\), let $h=h^\gamma$ be constructed from an arbitrary orientation on $\gamma$ as in \eqref{eq:phi_gamma}. Let $N\in\mathbb{N}$. In fact, we can use the proof of \cite[Theorem 4]{bordenave2026logarithmic} to obtain the estimate directly. We offer an alternative proof using a deterministic argument which is based directly on Lemma \ref{th:monof} and equivalent to the random labelling used in the proof of \cite[Theorem 4]{bordenave2026logarithmic}.

We construct the following voltage graph \(G^N_{\Gamma}\) over \(G\) with voltages in \(\Gamma \times \Z_N\) where
\begin{enumerate}[(i)]
            \item the vertex set is \(V(G)\times \Gamma \times \Z_N \),
            \item there is a (directed) edge between \((u,x,i),(v,y,j)\) whenever there is a directed edge \(e\) from \(u\) to \(v\) in \(G\), \(j=i+\phi(e) \operatorname{mod} N\) and \(y=xg_e\) in \(\Gamma\), where $g_e$ are given by the voltage labelling of $G_\Gamma$.
\end{enumerate} 
In other words, the voltage assignment to the directed edge $e\in\vec{E}(G)$ is $(g_e,\phi(e))\in\Gamma\times\Z_N$.

Let \(G^N:=G_{\Z_N}\) be the voltage graph over \(G\) whose vertex set is \(V(G)\times \Z_N\), such that there is a (directed) edge between \((u,i),(v,j)\) whenever there is a directed edge \(e\) from \(u\) to \(v\) in \(G\) and \(j=i+\phi(e) \operatorname{mod} N\). So, the voltage assignment to $e\in\vec{E}(G)$ is $\phi(e)$. We can also see \(G^N_\Gamma\) graph as a \(\Gamma\) voltage graph over \(G^N\) where the voltage assignment to the directed edge from $(u,i)$ to $(v,j)$ that is a lift of \(e\in \vec{E}(G)\) is given by $g_{e}$.

We show that \(G^N_{\Gamma}\) is isomorphic to \(N\) disjoint copies of \(G_\Gamma\). To this end we use the height function \(h=h^\gamma\) on \(G_\Gamma\) related to \(\gamma\) defined above. Define \[\iota: V(G_\Gamma)\times [N] \rightarrow V(G^N_{\Gamma}), \qquad (u,x,j) \rightarrow (u,x,h(u,x)+j\operatorname{mod} N).\] 

Note that \(V({G_\Gamma})\times [N]\) is the vertex set of a union of \(N\) disjoint copies of \(G_\Gamma\), viewed as a \(\Gamma\)-voltage graph on \(G\times [N]\) denoted $(G\times [N])_\Gamma$. It is direct to see the map is surjective. It is also injective on vertices since if \((u_1,x_1,j_1)\) has the same image as \((u_2,x_2,j_2)\), then 
\((u_1,x_1)=(u_2,x_2)\) and thus \(j_1\equiv j_2\operatorname{mod} N\), thus \(j_1=j_2\). By construction, it is easy to check that $\iota$ induces a local isomorphism between $(G\times [N])_\Gamma$ and $G_\Gamma^N$ and so it is also a global isomorphism of graphs as it is a bijection on vertices. Though this isomorphism does not intertwine with the action of \(\Gamma\) on both sides, Lemma \ref{lm;isodeterminedos} still provides for any interval \(I\subset \R\) that
\begin{equation}\label{eq:Nfactor}
    \mu_{(G \times [N])_\Gamma}(I)=N \mu_{G_\Gamma}(I)=\mu_{G^N_\Gamma}(I).
\end{equation}

The graph \(G^N_{\Gamma}\) has a natural labelling by $(u,x,i)\mapsto i$ and this is \(\Gamma\)-invariant over $G^N$. Due to our construction, \((u,x,i)\) is level if and only if \(u\in V(G\setminus\gamma)\). As with \eqref{eq:Nfactor}, for any interval \(I\subset \R\), we have 
\[N \mu_{(G\setminus\gamma)_\Gamma}(I)=\mu_{(G\setminus\gamma)^N_\Gamma}(I),\]
where \((G\setminus\gamma)^N:=(G\setminus\gamma)_{\Z_N}\) is the subgraph of \(G^N\) in the natural way. For \((u,x,i)\) such that \(u\in V(\gamma)\), if { \(i\neq \{0,1 \}\)}, then \((u,x,i)\) is prodigy, otherwise it is bad. So, by Proposition \ref{th:monof}, there exists \(\beta>0\),
such that for any  interval \(I\subseteq \R\), 
with \(|I|<\epsilon<\beta^{-N}\), we have 
$$
\mu_{G^N_\Gamma}(I) \leq \frac {2}{|V(G)|} + \frac{|V(G\setminus\gamma)|}{|V(G)|}\mu_{(G\setminus \gamma)^N_\Gamma}(I')
$$
for \(I'\) closed interval with \(|I'|=|I|^{1/(6N\log N)}\) and same center than $I$,
hence 
$$
\mu_{G_\Gamma}(I) \leq  \frac{2}{N|V(G)|}+ \frac{|V(G\setminus\gamma)|}{|V(G)|}\mu_{(G\setminus\gamma)_\Gamma}(I').
$$
Let \(C^1,\cdots,C^m\) be the connected components of \(G\setminus\gamma\). Then \(G_\Gamma\) is the disjoint union of \(C^i_\Gamma\). Then again by the definition of the spectral measure, \[|V(G\setminus\gamma)|\mu_{(G\setminus\gamma)_\Gamma}=\sum_{i}|V(C_i)|\mu_{C^i_\Gamma}.\]
By Lemma \ref{lemma:Hughesfree}, $C^i_\Gamma$ covers the maximal abelian cover of \(C^i\). Since for each $i$, $C_i$ has fewer edges than $G$ as $\gamma\neq\emptyset$, the induction hypothesis provides \(c, \alpha>0\) and \(0<\epsilon'_0(\gamma)<\delta\) such that for any
\(0<|I'|\leq \epsilon'_0(\gamma)\),
$$
\begin{aligned}
    \mu_{(G\setminus\gamma)_\Gamma}(I') &\leq \frac{1}{|V(G\setminus\gamma)|} \sum_{\theta \in I'\cap S_0} \mu_{(G\setminus\gamma)^{\ab}}(\{\theta\})+\frac{c}{\log ^{\alpha} |I'|^{-1}}\\
    &\leq \frac{1}{|V(G\setminus\gamma)|} \sum_{\theta \in I'\cap S_0} \mathrm{mult}_{\theta}\,(G\setminus\gamma)+\frac{c}{\log ^{\alpha} |I'|^{-1}}.
\end{aligned}
$$
For any \[|I|\leq \epsilon_0(\gamma):= \min \{\beta^{-N}, \epsilon'_0(\gamma)^{36}\},\]
choose{
\[N=\lfloor\frac{\log  |I|^{-1}}{\log \beta}\rfloor \wedge \lfloor \frac{1}{3}\frac{\log^{1/2} |I|^{-1}}{\log^{1/2} \epsilon_0(\gamma)^{-1}}\rfloor\geq 2.\]
Then \(|I| \leq \beta^{-N}\) and \(|I'|\leq \epsilon'_0(\gamma)\). Here for the sake of simplicity, we use a very rough bound \(6 N\log N \leq 9N^2\). For any \(0<c_0<1/2\), we have
\[\frac{1}{\log ^{\alpha} |I'|^{-1}}\lesssim \frac{N^\alpha\log^\alpha N}{\log ^{\alpha} |I|^{-1}}\lesssim  \frac{1}{\log ^{c_0\alpha} |I|^{-1}}\qquad \text{and}\qquad \frac{1}{N}\lesssim \frac{1}{\log^{1/2} |I|^{-1}}.\]
}
Moreover, since \(\epsilon'_0(\gamma)<\delta\), \(I \subset I'\) contain at most one element in \(S_0\), we have
\[\sum_{\theta \in I'\cap S_0} \mathrm{mult}_\theta (G\setminus\gamma)=\sum_{\theta \in I\cap S_0} \mathrm{mult}_\theta (G\setminus\gamma).\]
Hence there exists \(c(\gamma),c_0(\gamma)>0\), such that for any \(|I|\leq \epsilon_0(\gamma)\), 
\[\mu_{G_\Gamma}(I)\leq \frac{1}{|V(G)|}\cdot \sum_{\theta \in I\cap S_0} \mathrm{mult}_\theta (G\setminus\gamma)+\frac{c(\gamma)}{\log^{c_0(\gamma)}|I|^{-1}}.\]
\end{itemize}

We may now gather our estimates to conclude. For any \[|I|\leq \min_{\varnothing\neq \gamma \in \mathcal C(G)} \epsilon_0(\gamma) \wedge \min_{\theta_1\in S_1}\epsilon_1(\theta_1) \wedge \min_{\theta_2\in S_2}\epsilon_2(\theta_2),\]
there are the following four cases. Note that \(I\)
contains at most one element in \(S_0\).
\begin{itemize}
    \item \(I\) contains an element in \(S_1\). Then we perform Estimate I and get\[\mu_{G_\Gamma}(I)=0.\]
    \item \(I\) contains an element \(\theta_2\) in \(S_2\). Then we perform Estimate II and use 
    \[\mathrm{mult}_{\theta_2}(G)=\sum_{\theta \in S_0\cap I}\min_{\gamma\in \mathcal C(G)}\mathrm{mult}_\theta (G\setminus \gamma).\]

    \item \(I\) contains an element \(\theta_0\) in \(S_0\setminus (S_1\sqcup S_2)\). Choose \(\gamma_0 \neq \varnothing\) that is the minimizer
    of \(\mathrm{mult}_\theta(G\setminus \gamma)\). We then perform estimate III and use 
    \[\sum_{\theta\in S_0\cap I}\mathrm{mult}_{\theta}(G\setminus\gamma_0)=\mathrm{mult}_{\theta_0}(G\setminus\gamma_0)=\sum_{\theta \in S_0\cap I}\min_{\gamma\in \mathcal C(G)}\mathrm{mult}_\theta (G\setminus \gamma).\]

    \item \(I\) doesn't contain any element in \(S_0\), then pick any \(\gamma_0\neq \varnothing\) and perform estimate III. Note that now we have 
    \[\sum_{\theta \in S_0\cap I} \mathrm{mult}_\theta (G\setminus \gamma_0)=\sum_{\theta \in S_0\cap I}\min_{\gamma\in \mathcal C(G)}\mathrm{mult}_\theta (G\setminus \gamma)=0.\]
\end{itemize}
It concludes the proof.
\end{proof}

\subsection{The double formula and proof of Theorem \ref{conj}}

The proof of Theorem \ref{conj} is now   immediate.

\begin{proof}[Proof of Theorem \ref{conj}]
We start by proving the double formula \eqref{eq:sandwich}.
 In the proof of Theorem \ref{thm:lowerbound}, we showed that 
\[\min_{\gamma\in\mathcal{C}(G)}
\operatorname{mult}_\theta(G\setminus \gamma) \leq \kappa_\theta \leq \max_{S\in\mathcal A_\theta(G)}
\left(\operatorname{cc}S-|\partial S|\right) \leq |V|\mu_{G_{\Gamma}}(\theta).\]
Theorem \ref{thm:upperbound} now tells us that 
\[|V|\mu_{G_{\Gamma}}(\{\theta\})\leq \min_{\gamma\in\mathcal{C}(G)}
\operatorname{mult}_\theta(G\setminus \gamma).\]
Thus
\[\mu_{G_\Gamma}(\{\theta\})=\frac{1}{|V|}\max_{S\in\mathcal A_\theta(G)}
\left(\operatorname{cc}S-|\partial S|\right)= \frac{1}{|V |}\min_{\gamma\in\mathcal{C}(G)}
\operatorname{mult}_\theta(G\setminus\gamma).\]
This is precisely \eqref{eq:sandwich}.

It remains to prove the first claim of Theorem \ref{conj}. Let $I$ be a closed interval of sufficiently small length (as in Theorem \ref{thm:upperbound} and set 
$
s = \sum_{\theta \in I} \mu_{G^\ab} (\{ \theta \}),
$
where the sum is only over atoms $\theta$ of $G^\ab$.
From Theorem \ref{thm:upperbound} and $s \leq \mu_{G^\ab} (I)$, we find 
$$
\mu_{G_\Gamma} (I)  \leq  s +  \frac{c}{\log^\alpha(1/|I|)} \leq \mu_{G^\ab} (I) + \frac{c}{\log^\alpha(1/|I|)}.
$$
On the other hand, from \eqref{eq:sandwich}, we have $s = \sum_{\theta \in I} \mu_{G_\Gamma} (\{\theta\}) $. Hence, $s \leq \mu_{G_\Gamma} (I)$  and  Theorem \ref{thm:upperbound} now applied to $G_\Gamma = G^\ab$ gives
$$
\mu_{G^\ab} (I)  \leq  s +  \frac{c}{\log^\alpha(1/|I|)}\leq  \mu_{G_\Gamma} (I) + \frac{c}{\log^\alpha(1/|I|)},
$$
concluding the proof.
\end{proof}

Applied to $G_\Gamma = G^\uni$, the double formula \eqref{eq:sandwich} allows to recover Banks, Garza-Vargas and Mukherjee's criterion for \(\ell^2\)-eigenvalues of the universal cover. 

\section{A no-go result }\label{Section:eg}

In this section we show that Theorem \ref{conj} is almost sharp in the following sense. In particular, this gives a counter-example of the conjecture made in \cite{bzduvsek2022flat}.  The proof is based on a monotone labelling type of argument (which is slightly different from the standard one used in the last section and rather elementary).

\begin{proposition}\label{prop:sharp}
For any finitely generated group \(\Gamma\) 
such that there is a torsion free element in \([\Gamma,\Gamma]\), there exists 
a connected finite graph \(G\) with
\(b_1(\Gamma)=b_1(G)-1\)
such that
$$\operatorname{Spec}_{p} G_\Gamma \neq   \operatorname{Spec}_{p} G_{\Gamma^{\mathrm{ab}}}$$ with respect to the adjacency operator, where \(\operatorname{Spec}_{p}\) represents the point spectrum. 
\end{proposition}
}

\begin{proof}
Fix a finitely generated group 
\[
\Gamma=\langle s_1,\ldots,s_r\mid R\rangle\]
whose commutator group \([\Gamma,\Gamma]\) meets the requirements in the statement. Consider the bouquet \(\{\{v_0\}, \{\gamma_0, \cdots, \gamma_r\}\}\) with \(r+1\)-self loops and the following two voltage graphs over it. Pick a torsion free element \(g\in [\Gamma,\Gamma]\). Let \(\Gamma^\ab\) be the abelianisation of \(\Gamma\) with the abelianisation map \(\pi_\ab: \Gamma \rightarrow \Gamma^\ab\). The voltage assignments are respectively
    
          \[  \gamma_i \mapsto s_i\  (\pi_\ab(s_i)) \ \ \text{for}\ \  1\leq i\leq r \qquad \text{and} \qquad 
            \gamma_0 \mapsto g\ (\pi_\ab(g)=\mathrm{id}).
      \]
 The resulting graphs are the Cayley graphs of \(\Gamma\) and \(\Gamma^{\ab}\) with respect to the generating sets $\{s_i^\pm, g^\pm\}$ and $\{\pi_\ab(s_i)^\pm, \mathrm{id}\}$. Note that there is a self-loop at each vertex of the latter.

Now subdivide \(\gamma_0\) to obtain a cycle \(\gamma_0'\) of length \(4\) which turns the bouquet into a finite graph \(G\) with \(4\) vertices (\(4\) can be replaced by any even number larger than \(2\)). Conducting the same subdivision of the corresponding edge of the Cayley graphs on \(\Gamma\) and \(\Gamma^\ab\)
provides two covers \(G_1\) and \(G_2\) of $G$ with deck groups isomorphic to $\Gamma$ and its abelianization respectively.

We claim that \(0\) is an eigenvalue of \(G_2\) and not \(G_1\). First, any component in the preimage of \(\gamma_0'\) remains a \(4\)-cycle in \(G_2\) with one vertex on the cycle a cut vertex. One can then choose an eigenfunction on \(C_4\) of eigenvalue \(0\) that vanishes at this cut vertex and extend it by zero on the remaining vertices to obtain a finitely supported eigenfunction on \(G_2\) with eigenvalue $0$.

For $G_1$, the cycle \(\gamma_0'\) lifts to a collection of infinite paths since \(g\) is not a torsion element of $\Gamma$. Choose one of these path preimages 
\[\cdots \rightarrow v_{0;i}\rightarrow v_{1;i} \rightarrow v_{2;i} \rightarrow v_{3;i} \rightarrow v_{0;i+1} \rightarrow \cdots\]
where each \(v_{0;i}\) is a preimage of the bouquet vertex in the voltage graph before the subdivision of the loop $\gamma_0$. Together, these infinite paths contain \emph{every} vertex of \(G_1\)   and we use a monotone labelling type of argument even though these lines do not give a well-defined height function on \(G_1\) (which is possible since we are working with a possible eigenvalue \(0\)). Suppose that \(f\) were an eigenfunction for the adjacency operator corresponding to \(0\) on \(G_1\). Applying the adjacency operator to $f$ at \(v_{1;i}\) and \(v_{3;i}\) forces the relation \[f(v_{0;i})=-f(v_{2;i})=f(v_{0;i+1}),\] for every $i$. Since \(f\) is an \(\ell^2\) function, this forces \(f(v_{0;i})=f(v_{2;i})=0\) for every $i$ on every infinite path.

Similarly, the eigenfunction equation for $f$ at $v_{2;i}$ and $v_{0;i}$ (along with the fact that $f(v_{0;i})=0$ for all $i$ on every path) will force
    $$ f(v_{1;i})=-f(v_{3;i})\qquad \text{and} \qquad  f(v_{1;i})=-f(v_{3;i-1}),$$
for all $i$ on every path. Again, since $f$ is an $\ell^2$ function, this implies that \(f(v_{3;i})=0\) and hence also $f(v_{1;i})=0$ for all $i$ on every path. So, \(f\) is identically \(0\) and hence $0$ is not an atom of $G_1$.
\end{proof}

There is also a geometric perspective of the argument above. For example, let \[\Gamma=\langle a_1,b_1,a_2,b_2| [a_1,b_1][a_2,b_2]=e\rangle \] be the surface group of genus \(2\). Pick \(g=[a_1,b_1]=[a_2,b_2]^{-1}.\) For the one-holed torus, consider its standard \(1\)-skeleton with \(1\) vertex on the boundary, and \(3\) self-loops, one forming the boundary and the other two traversing around the hole and through the hole respectively. Take another copy of the one holed torus with the same \(1\)-skeleton and identify the boundary edge and also the two vertices, this gives a \(1\)-skeleton of the surface of genus \(2\). 

Now, subdivide the joining boundary edge by adding another \(3\) vertices along it. This embeds \(G\) into the closed surface \(\mathcal{S}_2\) of genus \(2\). Moreover, \(\gamma'_0\) is a separating curve in the surface. The graphs \(G_1\) and \(G_2\) correspond to the preimage of $G$ through the covering map of the universal and maximal abelian covers of \(\mathcal{S}_2\) respectively. The covering map from \(G_1\) factorises through \(G_2\) with deck group the surface group of genus \(2\) and a free abelian group of rank \(4\), respectively. Since \(\gamma_0'\) is separating, it lifts to loops in the maximal abelian cover but to infinite lines in the universal cover of the surface. One can make all edges geodesics segment with respect to a given metric on the surface \(\mathcal{S}_2\) of constant curvature. Lift this metric on \(\mathcal{S}_2\) to the universal cover of \(\mathcal{S}_2\) makes \(G_1\) a periodic graph embedded in the hyperbolic plane with all edges geodesic segments. This gives a counter-example of the conjecture made in \cite{bzduvsek2022flat} that for all hyperbolic lattices the atomic parts of $\mu_{G_\Gamma}$ and $  \mu_{G_{\Gamma^{\mathrm{ab}}}}$ are equal.

\appendix
\section{A converse Gallai-Edmonds theorem}

\label{appendix}
The generalised Gallai-Edmonds theorem tells that for each graph \(G\) with \(m_G(\theta)=0\), the bipartite graph \(G_\theta\) has positive surplus on the vertices corresponding to $\partial D_\theta(G)$, and vertices from the other side correspond to its critical components. We now show the following converse is also true. This generalises \cite[Theorem 3.2.3]{lovasz2009matching} in the classic setting.

Let \((X\sqcup Y, E)\) be a bipartite graph with positive surplus from \(X\). Let \(G\) be any graph built from \((X\sqcup Y, E)\) through the following steps:
\begin{itemize}
    \item Add edges between any vertices in $X$.

    \item Replace each $y\in Y$ by a \(\theta\)-critical graph \(K_y\). If $X\ni x\sim y$ is an edge in $E$, replace it with an edge joining $x$ to any vertex in $K_y$. 
    
    \item  For any graph \(H\) with \(\theta\) not a root of its generalised matching polynomial, we can add edges between \(V(H)\) and \(X\).
\end{itemize}

\begin{theorem}\label{conversedGallai-Edmonds}
    The \(\theta\)-critical components of \(G\) are exactly \(\{K_y\}_{y\in Y}\). Moreover, \(X\) is the set of \(\theta\)-special vertices of \(G\), that is, the boundary of the $\theta$-critical vertex set.
\end{theorem}

\begin{proof}
We first show that the multiplicity of $\theta$ as a root of the matching polynomial of $G$ is $|Y|-|X|$. To this end, we use \eqref{eq:multiedgerecursion} with $F=E$ the edge set of the starting bipartite graph $(X\sqcup Y,E)$, so that
    \[ m_G(x)=\sum_{M \in \mathcal{M}_G,\; E(M) \subset F} (-1)^{|E(M)|}\prod_{e\in \vec{E}(M)}w_e\cdot m_{G\setminus (F\cup V(M))}(x).\]
If \(M\) only contains vertices from at most $|X|-1$ of the $K_y$ components, then \(G\setminus (F\cup V(M))\) contains at least \(|Y|-|X|+1\) of the $K_y$ components which, because they are now isolated, are $\theta$-critical and have $\theta$-multiplicity equal to $1$ by Theorem \ref{thm:generalisedGallaiEdmonds}. Thus for all such \(M\), the multiplicity of \(m_{G\setminus (F\cup V(M))}\) at \(\theta\) is at least \(|Y|-|X|+1\). 

The remaining matchings $M$ visit \(|X|\) of the $K_y$ components and cover all vertices of \(X\). Since $(X\sqcup Y,E)$ is of positive surplus from \(X\), such a matching exists. For all such \(M\), components of \(G\setminus (F\cup V(M))\) are either the graphs \(H\) added in the third part of the construction, components \(K_{y}\) with one vertex removed, or the \(|Y|-|X|\) components \(K_y\) that \(M\) doesn't touch. By construction, \(\theta\) is not a root of the matching polynomial of the first two cases; thus, the multiplicity of \(m_{G\setminus (F\cup V(M))}\) at \(\theta\) is exactly \(|Y|-|X|\). Since all of these matchings $M$ give polynomial factors with the same sign \((-1)^{|E(M)|}\) and the polynomials themselves are monic, they don't cancel out; thus the multiplicity of \(m_G\) at \(\theta\) is exactly \(|Y|-|X|\).

We now show that the vertices in the $K_y$ are $\theta$-critical in $G$. To this end, fix some vertex \(u\in K_z\) for some $z\in Y$. Let \(F\) now be the edges between \(X\) and vertices from \(\bigsqcup_{y\in Y}K_y\setminus\{u\}\) so that by \eqref{eq:multiedgerecursion}, 
\[ m_{G\setminus\{u\}}(x)=\sum_{M \in \mathcal{M}_G,\; E(M) \subset F} (-1)^{|E(M)|}\prod_{e\in \vec{E}(M)}w_e\cdot m_{G\setminus( F\cup V(M)\cup \{u\})}(x).\]
Note that if \(M\) contains vertices from at most \(|X|-1\) of the \(K_y\) outside of \(K_z\), \(z\neq y\in Y\), then \(G\setminus( F\cup V(M)\cup \{u\})\) contains at least \(|Y|-|X|\) of the $K_y$ without any vertices removed, and so the multiplicity of \(m_{G\setminus( F\cup V(M)\cup \{u\})}\) at \(\theta\) is at least \(|Y|-|X|\). On the other hand, there always exists a matching \(M\) that contains vertices from \(|X|\) of the \(K_y\) components outside of \(K_z\) since the original bipartite graph has positive surplus from \(X\). Such a matching must covers all vertices from \(X\). A component of \(G\setminus( F\cup V(M)\cup \{u\})\) is then either  a component of some \(K_y\) after the removal of a vertex, one of the graphs \(H\), or one of the \(|Y|-|X|-1\) remaining untouched $K_y$. As before, this means that the multiplicity of \(m_{G\setminus( F\cup V(M)\cup \{u\})}\) at \(\theta\) is \(|Y|-|X|-1\). Together, this means that $m_{G\setminus\{u\}}$ has multiplicity $|Y|-|X|-1$ at $\theta$ and so $u$ is $\theta$-critical.

The same argument also shows that after deleting any vertex from the vertex sets of \(H\) or \(X\), the multiplicity is at least \(|Y|-|X|\) (actually \(|Y|-|X|+1\) in the latter case). This means that the vertices in these sets are either $\theta$-positive or $\theta$-neutral. Together, this means that the $K_y$ are the $\theta$-critical components of $G$ and $X$ forms the boundary.
\end{proof}

\bibliographystyle{abbrv}
\bibliography{multi}

\begin{tabular}{@{}l@{}}
Charles Bordenave \\
I2M, CNRS, Aix-Marseille Université, 
Marseille, France\\
\texttt{charles.bordenave@cnrs.fr}
\end{tabular}

\vspace{1em}

\begin{tabular}{@{}l@{}}
Wenbo Li \\
School of Mathematical Sciences,
University of Science and Technology of China, 
Hefei, China \\
\texttt{patlee@mail.ustc.edu.cn}
\end{tabular}

\vspace{1em}

\begin{tabular}{@{}l@{}}
Joe Thomas \\
Department of Mathematical Sciences,
Durham University, 
Durham, UK \\
\texttt{joe.thomas@durham.ac.uk}
\end{tabular}
\end{document}